\documentclass{article} % For LaTeX2e
\usepackage{macros}

\title{Simultaneous Perturbation Algorithms for Batch Off-Policy Search} 

\author[1]{Raphael Fonteneau \thanks{raphael.fonteneau@ulg.ac.be}}
\author[2]{Prashanth L A \thanks{prashanth.la@inria.fr}}
\affil[1]{\small Department of Electrical Engineering and Computer Science, University of Li\`{e}ge, Li\`{e}ge, Belgium}
\affil[2]{\small INRIA Lille - Nord Europe, Team SequeL, FRANCE.}
\date{}
\begin{document}

\maketitle
\thispagestyle{empty}
\pagestyle{empty}

%%%%%%%%%%%%%%%%%%%%%%%%%%%%%%%%%%%%%%%%%%%%%%%%%%%%%%%%%%%%%%%%%%%%%%%%%%%%%%%%
\begin{abstract}
 We propose novel policy search algorithms in the context of off-policy, batch mode reinforcement learning (RL) with continuous state and action spaces. Given a batch collection of trajectories, we perform off-line policy evaluation using an algorithm similar to that by  \cite{Fonteneau2010AISTATS}. Using this Monte-Carlo like policy evaluator, we perform policy search in a class of parameterized policies. We propose both first order policy gradient and second order policy Newton algorithms. All our algorithms incorporate simultaneous perturbation estimates for the gradient as well as the Hessian of the cost-to-go vector, since the latter is unknown and only biased estimates are available. We demonstrate their practicality on a simple 1-dimensional continuous state space problem.
\end{abstract} 

\section{Introduction}
\label{section:introduction}
This paper stands within the field of optimal control in the context of  infinite horizon discounted cost Markov decision processes (MDPs) \cite{BertsekasT96}. More specifically, this paper addresses the batch mode setting \cite{Ernst2005,Fonteneau2011Thesis}, where we are given a set of noisy trajectories of a system without access to any model or simulator of that system. More formally, we are given a set of $n$ samples (also called transitions) $\{(x^l, u^l, c^l, y^l)\}_{l=1}^n$, where,  for every $l \in \{1, \ldots, n\}$, the 4-tuple $(x^l, u^l, c^l, y^l)$ denotes the state $x^{l}$, the action $u^{l}$, a (noisy) cost received in $(x^{l}, u^{l})$ and a (noisy) successor state reached when taking action $u^{l}$ in state $x^{l}$. The samples are generated according to some unknown policy and the objective is to develop a (off-policy) control scheme that attempts to find a near-optimal policy using this batch of samples.

For this purpose, we first parameterize the policy and hence the cost-to-go, denoted by $J^\theta(x_0)$. Here $\theta$ is the policy parameter, $x_0$ is a given initial state and $J^\theta(x_0)$ is the expected cumulative discounted sum of costs under a policy governed by $\theta$ (see \eqref{performance_criterion}). Note that the policy parameterization is not constrained to be linear. We develop algorithms that perform descent using estimates of the cost-to-go $J^\theta(x_0)$.  For obtaining these estimates from the batch data, we extend a recent algorithm proposed for finite horizon MDPs \cite{Fonteneau2010AISTATS}, to the infinite horizon, discounted setting. The advantage of this estimator, henceforth referred to as MFMC, is that it is off-policy in nature, computationally tractable and consistent under Lipschitz assumption on the transition dynamics, cost function and policy.  Moreover, it does not require the use of function approximators, but only needs a metric on the state and action spaces.

Being equipped with the MFMC policy evaluator that outputs an estimate of the cost-to-go $J^\theta(x_0)$ for any policy parameter $\theta$, the requirement is for a control scheme that uses these estimated values to update the parameter $\theta$ in the negative descent direction. However, closed form expressions of the gradient/Hessian of the cost-to-go are not available and MFMC estimates possess a non-zero bias. To alleviate this, we employ the well-known simultaneous perturbation principle (cf. \cite{Bhatnagar13SR}) to estimate the gradient and Hessian, respectively, of $J^\theta(x_0)$ using estimates from MFMC and propose two first order and two second order algorithms. Our algorithms are based on two popular simultaneous perturbation methods - Simultaneous Perturbation Stochastic Approximation (SPSA) \cite{spall92multivariate} and Smoothed Functional \cite{katkovnik1972convergence}. 

The first-order algorithms perform gradient descent using either SPSA or SF estimates to update the policy parameter. On the other hand, the second order algorithms incorporate a Newton step by estimating the gradient as well as the Hessian of the cost-to-go $J^\theta(x_0)$ using SPSA or SF. 
%\todo[inline]{Commented risk variant. Should go to the techreport}
% We also discuss the extension of our algorithms to incorporate risk-sensitive criteria along the lines of \cite{prashanth2013actor}. In particular, we incorporate the variability of the return and develop a two-timescale algorithm that optimizes the expected return while having bounds on variance of the return. 
We demonstrate the empirical usefulness of our algorithms on a simple 1-dimensional continuous state space problem.

%\todo{I removed the summary of contributions as I found it to be repetitive. However, I cant accomodate the following paragraph easily into the introduction.}

To the best of our knowledge, the algorithms presented in this paper are the first to solve batch, off-policy stochastic control in continuous state and action spaces without using function approximators for evaluating policies. Our approach only requires (i) a (random) set of trajectories, (ii) metrics on the state and action spaces, and (iii) a set of parameterized policies.

%%%%%%%%%%%%%%%%%%%%%%%%%%%%%%%%%%%%%%%%%%%%
\section{Related work}
\label{section:related_work}
 
The work presented in this paper mainly relates to two fields of research: batch mode reinforcement learning and policy gradient methods.

%\todo{Cut a few refs to make space}
Genesis of batch mode RL may be found in the work of \citep{Bradtke1996}, where the authors use least-squares techniques in the context of temporal difference (TD) learning methods for estimating the return of control policies. This approach has been extended to the problem of optimal control by \citep{Lagoudakis2003}. Algorithms similar to value iteration have also been proposed in the batch mode RL setting and the reader is referred to the works of \citep{Ormoneit2002} (using kernel approximators) or \citep{Ernst2005} (using ensembles of regression trees) and \citep{Riedmiller2005} (using neural networks). 
More recently, new batch mode RL techniques have been proposed by \citep{Fonteneau2013ANOR} and this does not require the use of function approximators for policy evaluation. Our policy evaluator is based on the Monte Carlo-like technique proposed by \citep{Fonteneau2013ANOR}.

%\todo{Not sure if all of these refs - \cite{Grondman2012,Busoniu2011cross,Schmidhuber1998,konda2003onactor} are needed?}
Policy gradient methods \citep{Bartlett2001} can be seen as a subclass of direct policy search techniques \citep{Schmidhuber1998,Busoniu2011cross} that aim at  finding a near-optimal policy within a set of parameterized policies.  Actor-critic algorithms are relevant in this context and the reader is referred to works by \citep{konda2003onactor,bhatnagar2009natural,Grondman2012} and the references therein.
The actor-critic algorithms mentioned above work in an approximate dynamic programming setting. In other words, owing to the high-dimensional state spaces encountered often in practice, the algorithms approximate the value function with a (usually linear) function approximation architecture. Thus, the quality of the policy obtained by the algorithms are contingent upon the quality of the approximation architecture and selection of approximation architecture is in itself a hot topic of research in RL. In contrast, we employ a policy evaluation technique which does not resort to function approximation for the value function and works with a Monte Carlo like scheme instead. 
%Instead, working in a off-policy setting and assuming certain Lipschitz conditions on the underlying MDP, we establish that our policy evaluator returns an estimate of the expected return $J^\theta(x_0)$ with bounded bias. 

%%%%%%%%%%%%%%%%%%%%%%%%%%%%%%%%%%%%%%%%%%%%
\section{The Setting}
\label{section:setting}
We consider a stochastic discrete-time system
with state space $\mathcal X \subset \mathbb R^{d_{\mathcal X}}$, $d_{\mathcal X} \in \mathbb N$ and action space $\mathcal U \subset \mathbb R^{d_{\mathcal U}}$, $d_{\mathcal U} \in \mathbb N$. The dynamics of this system is governed by:
\begin{align*}
x_{t+1} = f\left(x_t,u_t,w_t \right), \qquad \forall t \in \mathbb N
\end{align*}
where $x_t$ and $u_t$ denote the state and action at time $t \in \mathbb N$, while $w_t \in \mathcal W$ denotes a random disturbance  drawn according to a probability distribution $p_{\mathcal W}(\cdot)$. Each system transition from time $t$ to $t+1$ incurs an instantaneous cost $c\left(x_t,u_t,w_t\right)$.
We assume that the cost function is bounded and translated into the interval $[0,1]$. 

Let $\mu : \mathcal X \rightarrow \mathcal U $ be a control policy that maps states to actions. 
In this paper, we consider a class of policies parameterized by $\theta \in \Theta$, i.e., $\mu^{\theta} :  \mathcal X \rightarrow \mathcal U$. We assume that $\Theta$ is a compact and convex subset of $\mathbb R^N, N \in \mathbb N$. Since a policy $\mu$ is identifiable with its parameter $\theta$, we shall use them interchangeably in the paper. 

The classical performance criterion  for evaluating a policy $\mu$ is its (expected) cost-to-go, which is the discounted sum of costs that an agent receives, while starting from a given initial state $x$ and then following a  policy $\mu$, i.e.,

\begin{align}
J^\mu(x_0) &= \mathbb  E \left[     \sum_{t=0}^{\infty}\gamma^t c(x_t,\mu(x_t),w_t) \mid x_0, \mu  \right],
\label{performance_criterion}\\
\mbox{ where } x_{t+1} &=  f(x_t,\mu(x_t),w_t) \text{ and }  
w_t \sim p_{\mathcal W}(\cdot),  \forall t \in \mathbb N.\nonumber
\end{align}
In the above, $\gamma \in (0,1)$ denotes the discount factor.

In a batch mode RL setting, the objective is to find a policy that minimizes the cost-to-go $J^\mu(x_0)$. However, the problem is challenging since the functions $f$, $c$ and $p_{\mathcal W}(\cdot)$ are unknown (not even accessible to simulation). Instead, we are provided with a batch collection of $n \in \mathbb N \setminus \{ 0\}$ one-step system transitions $\mathcal F_n$, defined as 
$${\mathcal  F_{n}} = \left\{ \left(x^l,u^l, c^l,y^l \right)\right\}_{l=1}^{n},$$ where $c^l := c\left(x^{l},u^{l},w^{l}\right)$ is the instantaneous cost and $y^{l}  :=  f\left(x^{l},u^{l},w^{l}\right)$ is the next state. Here, both $c^l$ and $y^l$ are governed by the disturbance sequence $w^{l} \sim p_{\mathcal W}(\cdot)$, for all $l \in \{ 1, \ldots, n \}$.

The algorithms that we present next incrementally update the policy parameter $\theta$ in the negative descent direction using either the gradient or Hessian of $J^{\theta}(x_0)$. The underlying policy evaluator that provides the cost-to-go inputs for any $\theta$ is based on MFMC, while the gradient/Hessian estimates are based on the principle of simultaneous perturbation (\cite{Bhatnagar13SR}).

%%%%%%%%%%%%%%%%%%%%%%%%%%%%%%%%%%%%%%%%%%%%
\section{Algorithm Structure}
\label{section:structure}
In a deterministic optimization setting, an algorithm attempting to find the minima of the cost-to-go $J^\theta(x_0)$ would update the policy parameter in the descent direction as follows:
\begin{equation}
\label{eq:theta_descent_det}
\theta_i(t+1) = \Gamma_i( \theta(t) - a(t) A_t^{-1} \nabla_\theta J^\theta(x_0)),
\end{equation}
where $A_t$ is a positive definite matrix and $a(t)$ is a step-size that satisfies standard stochastic approximation conditions: $\sum\limits_t a(t) = \infty$ and $\sum\limits_t a(t)^2 < \infty$.  Further, $\Gamma(\theta) = (\Gamma_1(\theta_1),\ldots,\Gamma_N(\theta_N))$ is a projection operator that projects the iterate $\theta$ to the nearest point in the set $\Theta \in \R^N$. The projection is necessary to ensure stability of the iterate $\theta$ and hence the overall convergence of the 
scheme \eqref{eq:theta_descent_det}.

For the purpose of obtaining the estimate of the cost-to-go vector $J^\theta(x_0)$ for any $\theta$, we adapt the MFMC (for Model-Free Monte Carlo) estimator proposed by \cite{Fonteneau2010AISTATS}) to our (infinite-horizon discounted) setting\footnote{Besides being adapted to the batch mode setting, the MFMC estimator also has the advantage of having a linear computational complexity and consistency properties (see Section \ref{section:MFMC}).}. The MFMC estimator works by rebuilding (from one-step transitions taken in $\mathcal F_n$) artificial trajectories that emulate the trajectories that could be obtained if one could do Monte Carlo simulations. An estimate $\hat J^\theta$ of the cost-to-go $J^\theta$ is obtained by averaging the cumulative discounted cost of the rebuilt artificial trajectories. 

Using the estimates of MFMC, it is necessary to build a higher-level control loop to update the parameter $\theta$ in the descent direction as given by \eqref{eq:theta_descent_det}. However, closed form expressions of the gradient and the Hessian of $J^\theta(x_0)$ are not available and instead, we only have (biased) estimates of $J^\theta(x_0)$ from MFMC. Thus, the requirement is for a simulation-optimization scheme that approximates the gradient/Hessian of $J^\theta(x_0)$ using estimates from MFMC. 

Simultaneous perturbation methods \cite{Bhatnagar13SR} are well-known simulation optimization schemes that perturb the parameter uniformly in each direction in order to find the minima of a function observable only via simulation. These methods are attractive since they require only two simulations irrespective of the parameter dimension.
Our algorithms are based on two popular simultaneous perturbation methods - Simultaneous Perturbation Stochastic Approximation (SPSA) \cite{spall92multivariate} and Smoothed Functional \cite{katkovnik1972convergence}.
The algorithms that we propose mainly differ in the choice of $A_t$ in \eqref{eq:theta_descent_det} and the specific simultaneous perturbation method used:
\begin{description}
 \item[MCPG-SPSA.] Here $A_t = I$ (identity matrix). Thus, MCPG-SPSA is a first order scheme that updates the policy parameter in the descent direction. Further, the gradient $\nabla_\theta J^\theta(x_0)$ is estimated using SPSA. 
 \item[MCPG-SF.] This is the Smoothed functional (SF) variant of MCPG-SPSA.
\item[MCPN-SPSA.] Here $A_t=\nabla^2 J^{\theta}(x_0)$, i.e., the Hessian of the cost-to-go. Thus, MCPN is a second order scheme that update the policy parameter using a Newton step. Further, the gradient/Hessian are estimated using SPSA.
\item[MCPN-SF.] This is the SF variant of MCPN-SPSA.
\end{description}
% We also provide a variant of MCPN where the inverse of the Hessian is directly estimated using Woodbury's identity. 

As illustrated in Fig. \ref{fig:algorithm-flow}, our algorithms operate on the principle of simultaneous perturbation and involve the following steps:\\
  \begin{inparaenum}[\bfseries(i)]
   \item estimate, using MFMC, the cost-to-go for two perturbation sequences $\theta(t)+p_1(t)$ and $\theta(t)-p_2(t)$;\\
   \item obtain the gradient/Hessian estimates (see \eqref{eq:gradestimate}--\eqref{eq:hessianestimate}) from the cost-to-go values $J^{\theta(t)+p_1(t)}(x_0)$ and $J^{\theta(t)+p_2(t)}(x_0)$;\\
   \item update the parameter $\theta$ in the descent direction using the gradient/Hessian estimates obtained above.\\
  \end{inparaenum}
The choice of perturbation sequences $p_1(t)$ and $p_2(t)$ is specific to the algorithm (see Sections \ref{section:first-order} and \ref{section:second-order}).
% Note that for MCPG algorithm, $p_1(t) = \delta\Delta(t)$ and $p_2(t) = - \delta\Delta(t)$, whereas for MCPN algorithm $p_1(t) = \delta\Delta(t)+\delta\hat\Delta(t)$ and $p_2(t) = \delta\Delta(t)$. Here $\delta>0$ is a small fixed constant. 

\begin{algorithm}[t] 	
\begin{algorithmic}
\STATE {\bf Input:}  $\theta_0$, initial parameter vector; $\delta>0$; $\Delta$; 
\STATE MFMC($\theta$), the model free Monte Carlo like policy evaluator 
\FOR{$t = 0,1,2,\ldots$}
\STATE Call  MFMC($\theta(t)+p_1(t)$)% and observe $J^{\theta(t)+p_1(t)}(x_0)$ 
\STATE Call MFMC($\theta(t)+p_2(t)$)% and observe $ J^{\theta(t)+p_2(t)}(x_0)$
\STATE Compute $\theta(t+1)$ (Algorithm-specific)
\ENDFOR
\STATE {\bf Return} $\theta(t)$
\end{algorithmic}
\caption{Structure of our algorithms.}
\label{alg:structure}
\end{algorithm}

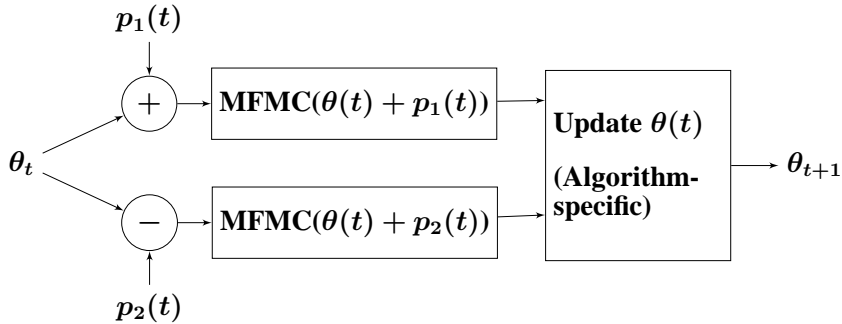
\begin{figure}[th]
\centering
\tikzstyle{block} = [draw, fill=white, rectangle,
   minimum height=3em, minimum width=6em]
\tikzstyle{sum} = [draw, fill=white, circle, node distance=1cm]
\tikzstyle{input} = [coordinate]
\tikzstyle{output} = [coordinate]
\tikzstyle{pinstyle} = [pin edge={to-,thin,black}]
\scalebox{0.9}{\begin{tikzpicture}[auto, node distance=2cm,>=latex']
% We start by placing the blocks
\node (theta) {\large$\bm{\theta_t}$};
\node [sum, above right=0.4cm of theta, xshift=1cm] (perturb) {\large$\bm{+}$};
\node [sum, below right=0.4cm of theta, xshift=1cm] (perturb1) {\large$\bm{-}$};
\node [above=0.5cm of perturb] (noise) {\large$\bm{p_1(t)}$};
\node [below=0.5cm of perturb1] (noise1) {\large$\bm{p_2(t)}$};    
\node [block, right=0.5cm of perturb] (psim) {\large\bf {MFMC($\bm{\theta(t)+p_1(t)}$)}}; 
\node [block, right=0.5cm of perturb1] (sim) {\large\bf {MFMC($\bm{\theta(t)+p_2(t)}$)}}; 
\node [block, below right=1cm of psim, minimum height=8em, yshift=1.74cm,text width=2.5cm] (update) {\large\bf{Update $\bm{\theta(t)}$}\\[2ex]\large\bf{(Algorithm-specific)}};
\node [right=0.7cm of update] (thetanext) {\large$\bm{\theta_{t+1}}$};

\draw [->] (perturb) --  (psim);
\draw [->] (perturb1) --  (sim);
\draw [->] (noise) -- (perturb);
\draw [->] (noise1) -- (perturb1);
\draw [->] (psim) -- %node {$\hat J^{\theta(t)+p_1(t)}(x_0)$}
(update.145);
\draw [->] (sim) --  %node {$\hat J^{\theta(t)+p_2(t)}(x_0)$} 
(update.208);
\draw [->] (update) -- (thetanext);
\draw [->] (theta) --   (perturb);
\draw [->] (theta) --   (perturb1);
\end{tikzpicture}}
\caption{Overall flow of simultaneous perturbation algorithms.}
\label{fig:algorithm-flow}
\end{figure}

\begin{remark}
From a theoretical standpoint, the setting considered here is of deterministic optimization and the estimates from MFMC have non-zero, albeit bounded, non-stochastic bias for a given sample of transitions. This is unlike earlier work on SPSA, which mostly feature a stochastic noise component that is zero-mean. While we establish bounds on the bias of MFMC (see Lemmas 1 and 2 in the Appendix), it is a challenge to establish asymptotic convergence and in this regard, we note the difficulties involved in Section \ref{sec:conv-mcpg}.
\end{remark}

%%%%%%%%%%%%%%%%%%%%%%%%%%%%%%%%%%%%%%%%%%%%
\section{MFMC Estimation of a Policy}
\label{section:MFMC}

%\todo{Rewrite this section to minimize overlap with earlier paper}
For the purpose of policy evaluation given a batch of samples, we adapt the Model-free Monte Carlo estimator (MFMC) algorithm, proposed by \cite{Fonteneau2010AISTATS}, to an infinite horizon discounted setting.
% The pseudocode of the resulting algorithm, also referred to MFMC, is presented in Algorithm \ref{algo}.

From a sample of transitions $\mathcal F_n$, the MFMC estimator rebuilds $p \in \mathbb N \setminus \{ 0 \}$ (truncated) artificial  trajectories. These artificial trajectories are used as approximations of $p$ trajectories that could be generated by simulating the policy $\mu^{\theta}$ we want to evaluate. The final MFMC estimate $\hat J^{\theta}(x_0)$ is obtained  by averaging  the  cumulative discounted  costs  over  these truncated artificial trajectories. 

The trajectories here are rebuilt in a manner similar to the procedure outlined by \cite{Fonteneau2010AISTATS}. However, in our (infinite horizon) setting, the horizon needs to be truncated for rebuilding the trajectories. To this end, we introduce a truncation parameter $T$ that defines the length of the rebuilt trajectories. To limit the looseness induced by such a truncation, the value of the parameter $T$ should be chosen as a function of the discount factor $\gamma$, for instance, $T= \Omega\left( \frac{1}{1-\gamma}  \right)$.
\begin{algorithm} 	
\begin{algorithmic}
\STATE {\bf Input:}  $\mathcal F_{n},  \mu^{\theta}(.,.),  x_{0}, d(.,.), T, p$
\STATE $\mathcal G$: current set of not yet used   one-step transitions   in $\mathcal F_{n}$; Initially, 
$\mathcal G \leftarrow \mathcal F_{n}$;
\FOR {$i=1$ to $p$} 
\STATE $t \leftarrow 0$; $x_{t}^{i} \leftarrow x_{0}$;
\WHILE { $t < T$ }
\STATE $u_{t}^{i} \leftarrow \mu^{\theta}\left(x_{t}^{i}\right)$;
\STATE 
  $\mathcal H  \leftarrow \underset {(x,u,c,y) \in \mathcal G} {\arg\min} \quad  d\left((x,u), \left(x_{t}^{i},   u_{t}^{i}\right)\right)$;
\STATE $l_{t}^{i} \leftarrow$ lowest index in $\mathcal F_{n}$ of the   transitions  that  belong  to  $\mathcal H$;
\STATE $t \leftarrow  t+1$; $x_{t}^{i} \leftarrow y^{l^i_t}$;
\STATE $\mathcal G \leftarrow   \mathcal G \setminus  \left\{\left(x^{l^i_t},u^{l^i_t},c^{l^i_t},
  y^{l^i_t}\right)\right\}$;
\ENDWHILE
\ENDFOR
\STATE {\bf Return} $ \hat J^\theta\left(x_0 \right) = \frac {1} {p} \sum_{i=1}^{p}
\sum_{t=0}^{T-1} \gamma^t c^{l_{t}^{i}} .$
\end{algorithmic}
\caption{MFMC algorithm.\label{algo}}
\end{algorithm}
The MFMC estimation can be computed using the algorithm provided in Algorithm \ref{algo}.
\begin{definition}[Model-free Monte Carlo Estimator]
\begin{align}
\hat J^\theta\left(x_0 \right) & =  \frac {1} {p} \sum_{i=1}^{p}
\sum_{t=0}^{T-1} \gamma^t c^{l_{t}^{i}} .  \nonumber
\end{align}
where $\left\{l_{t}^{i} \right\}_{i=1,  t=0}^{i=p,  t=T-1}$ denotes the set of indices of the transitions selected by the MFMC algorithm (see Algorithm \ref{algo}).
\end{definition}
Note  that  the  computation of  the  MFMC  estimator $\hat J^\theta\left(x_0 \right)$  has  a linear  complexity  with respect to the cardinality $n$  of $\mathcal F_{n}$, the number of artificial trajectories $p$ and the optimization horizon $T$. 

\begin{remark}
Through Lemmas 1 and 2 in the Appendix, we bound the distance between the MFMC estimate $\hat J^\theta\left(x_0 \right)$ and the true cost-to-go $J^\theta(x_0)$ in expectation and high probability, respectively. 
\end{remark}
% Further, we also provide the analysis of the MFMC estimator after incorporating the necessary enhancements owing to the infinite horizon discounted setting. 

%%%%%%%%%%%%%%%%%%%%%%%%%%%%%%%%%%%%%%%%%%%%
\section{Algorithms}
\label{section:algorithms}

% In this section, we describe the first and second order schemes, MCPG and MCPN, respectively for control in a batch mode RL setting. 
% As mentioned before, the first algorithm is a gradient descent scheme while the second is based on the Newton method.

% These tune the policy parameter $\theta$ in the ascent direction using SPSA estimates of the gradient and Hessian, respectively. 
\subsection{First order algorithms}
\label{section:first-order}
\subsubsection{Gradient estimates}
\textbf{SPSA} based estimation of the gradient of the cost-to-go is illustrated as follows: For the simple case of a scalar parameter $\theta$,  
  \begin{align}
\frac{d  J^{\theta}}{d\theta} \approx
\left(\dfrac{ J^{\theta+\delta} -  J^{\theta}}{\delta}\right).
\label{eq:scalar-spsa}
\end{align}
The correctness of the above estimate can be seen by first $J^{\theta+\delta}$ and $J^{\theta-\delta}$ around $\theta$ using a Taylor expansion as follows:
  \begin{align*}
   &J^{\theta+\delta} =  J^{\theta} + \delta \frac{d  J^{\theta}}{d\theta} +  O(\delta^2),
   J^{\theta-\delta} =  J^{\theta} - \delta \frac{d  J^{\theta}}{d\theta}  +  O(\delta^2).\\
   &\text{Thus, }\dfrac{ J^{\theta+\delta} -  J^{\theta-\delta}}{2\delta} = \dfrac{d  J^{\theta}}{d\theta} + O(\delta).
\end{align*}
From the above, it is easy to see that the estimate \eqref{eq:scalar-spsa} converges to the true gradient $\dfrac{d  J^{\theta}}{d\theta}$ in the limit as $\delta \rightarrow 0$.

The above idea of simultaneous perturbation can be extended to a vector-valued parameter $\theta$ by perturbing each co-ordinate of $\theta$ uniformly using Rademacher random variables. The resulting SPSA based estimate of the gradient $\nabla_{\theta} J^{\theta}(x_0)$ is as follows:
\begin{align}
\label{eq:gradestimate}
 \nabla_{\theta_i} J^{\theta}(x_0) \approx  
\frac{J^{\theta +\delta\Delta}(x_0) - J^{\theta-\delta\Delta}(x_0)}{2\delta\Delta_i},
\end{align}
where $\Delta = (\Delta_1,\ldots,\Delta_N)^T$ with each $\Delta_i$ being Rademacher random variables. 

\textbf{SF} based estimation of the gradient of the cost-to-go is given by
\begin{equation}
\label{eq:sf-grad-estimate}
 \nabla_{\theta_i} J^{\theta}(x_0) \approx \dfrac{\Delta_i}{\delta} 
\left( J^{\theta+\delta\Delta}(x_0) -  J^{\theta-\delta\Delta}(x_0) \right ),
\end{equation}
where $\Delta$ is a $(|N|)$-vector of independent $\N(0,1)$ random variables. 

\subsubsection{MCPG-SPSA and MCPG-SF algorithms}
\label{sec:mcpg}
On the basis of the gradient estimate in \eqref{eq:gradestimate}--\eqref{eq:sf-grad-estimate}, the SPSA and SF variants update the policy parameter $\theta$ as follows: For all $t \ge 1$, update
\begin{align}
\label{eq:spsa-update-rule}
\text{\bf SPSA: } &\theta_{i}(t+1)  =  \Gamma_i \bigg( \theta_{i}(t) - a(t)
\frac{\hat J^{\theta(t)+\delta\Delta(t)}(x_0) - \hat J^{\theta(t)-\delta\Delta(t)}(x_0)}{2\delta \Delta_{i}(t)}
 \bigg),\\
\label{eq:sf-update-rule}
 \text{\bf SF: } &\theta_{i}(t+1)  =  \Gamma_i \bigg( \theta_{i}(t) - a(t)
\frac{\Delta_{i}(t)}{2\delta}(\hat J^{\theta(t)+\delta\Delta(t)}(x_0) - \hat J^{\theta(t)-\delta\Delta(t)}(x_0))
 \bigg),
\end{align}
for all $i = 1, 2, \dots, N$.
In the above,\\ 
\begin{inparaenum}[\bfseries(i)]
\item $\delta>0$ is a small fixed constant and $\Delta(t)$ is a $N$-vector of independent Rademacher random variables for SPSA and standard Gaussian random variables for SF; \\
\item $\hat J^{\theta(t)+\delta\Delta(t)}(x_0)$ and $\hat J^{\theta(t)-\delta\Delta(t)}(x_0)$ are the MFMC policy evaluator's estimates of the cost-to-go corresponding to the parameters $\theta+\delta\Delta$ and $\theta-\delta\Delta$, respectively.\\
\item $\Gamma(\theta) = (\Gamma_1(\theta_1),\ldots,\Gamma_N(\theta_N))^T$ is an operator that projects the iterate $\theta$ to the closest point in a compact and convex set $\Theta \in \R^N$;
\item $\{a(t),t \ge 1\}$ is a step-size sequence that satisfies the standard stochastic approximation conditions.
\end{inparaenum}

\begin{remark}
A standard approach to accelerate stochastic approximation schemes is to use Polyak-Ruppert averaging, i.e., to return the averaged iterate $\bar \theta_{t+1} := \sum\limits_{s=1}^{t} \theta_s$ instead of $\theta_t$. 
% at the end of the optimization procedure. 
% It will be seen in the Section \ref{section:results} that bigger step-sizes i.e.,$a(t) = \Theta(t^{-alpha})$, $\alpha \in \left(1/2,1\right)$ coupled with iterate averaging results in rate of convergence $O(t^{-1/2})$. On the other hand, with a step-size of $a(t) = \Theta(t^{-1})$, achieving the same rate requires knowledge about the structure of $J^\theta(x_0)$.
\end{remark}

\subsection{Second order algorithms}
\label{section:second-order}
For the second order methods, we also need an estimate of the Hessian $\nabla_{\theta}^2 J^{\theta}(x_0)$, in addition to the gradient.
\subsubsection{Hessian estimates}
\textbf{SPSA} based estimate of the Hessian $\nabla_{\theta}^2 J^{\theta}(x_0)$ is as follows:
\begin{align}
\label{eq:hessianestimate}
 \nabla_{\theta_i}^2 J^{\theta}(x_0) \approx 
\frac{J^{\theta +\delta\Delta+\delta\widehat{\Delta}}(x_0) - J^{\theta +\delta\Delta}(x_0)}{\delta^2\Delta_i \widehat{\Delta}_i},  
\end{align}
where $\Delta$ and $\widehat\Delta$ represent $N$-vectors of Rademacher random variables\footnote{For a precise statement of the asymptotic correctness of the gradient and Hessian estimates, see Lemmas \ref{lemma:spsa-n}--\ref{lemma:sf-n} in Appendix \ref{sec:convergence}.}.

\textbf{SF} based estimate of the Hessian $\nabla_{\theta}^2 J^{\theta}(x_0)$ is as follows:
\begin{align}
\label{eq:sf-hessian-estimate}
 \nabla_{\theta_i}^2 J^{\theta}(x_0) \approx 
\frac{1}{\delta^2} \bar{H}(\Delta)\big(J^{\theta +\delta \Delta}(x_0) + J^{\theta-\delta\Delta}(x_0)\big),
\end{align}
where $\Delta$ is a $N$ vector of independent Gaussian $\N(0,1)$ random variables and $\bar{H}(\Delta)$ is a $N \times N$ matrix defined as
\begin{equation}
\label{H-bar}
\bar{H}(\Delta) \stackrel{\triangle}{=}
\left[
\begin{array}{cccc}
\big(\Delta_{1}^2-1\big) & \Delta_{1}\Delta_{2} & \cdots &
\Delta_{1}\Delta_{N}\\
\Delta_{2}\Delta_{1}& \big(\Delta_{2}^2-1\big) & \cdots &
\Delta_{2}\Delta_{N}\\
\cdots & \cdots & \cdots & \cdots \\
\Delta_{N}\Delta_{1} & \Delta_{N}\Delta_{2} & \cdots &
\big(\Delta_{N}^2-1\big)
\end{array}
\right].
\end{equation}

\subsubsection{MCPN-SPSA and MCPN-SF algorithms}
\label{sec:mcpn}
Let $H(t) = [H_{i,j}(t)]_{i = 1,j = 1}^{|N|, |N|}$ denote the estimate of the Hessian w.r.t. $\theta$ of the cost-to-go $J^\theta(x_0)$ at instant $t$, with $H(0)=\omega I$ for some $\omega > 0$. On the basis of \eqref{eq:hessianestimate}, MCPN-SPSA would estimate the individual components $H_{i, j}(t)$ as follows:
For all $t \ge 1$, $i,j\in \{1,\ldots,N\}$, $i\le j$, update
\begin{align}
&H_{i, j}(t + 1) =  H_{i, j}(t) + a(t) \bigg ( \dfrac{\hat J^{\theta(t)+\delta\Delta(t) + \delta\widehat\Delta(t)}(x_0) - \hat J^{\theta(t)+\delta\Delta(t)}(x_0)}{\delta^2 \Delta_{j}(t) \widehat\Delta_{i}(t)} - H_{i, j}(t) \bigg ),
\end{align}
and for $i > j$, set $H_{i, j}(t+1) = H_{j, i}(t+1)$. In the above, $\delta>0$ is a small fixed constant and $\Delta(t)$ and $\hat\Delta(t)$ are $N$ vectors of Rademacher random variables.
Now form the Hessian inverse matrix $M(t) = \Upsilon(H(t))^{-1}$. The operator $\Upsilon(\cdot)$ ensures that the Hessian estimates stay within the set of positive definite and symmetric matrices. This is a standard requirement in second-order methods (See \cite{gill1981practical} for one possible definition of $\Upsilon(\cdot)$).               
Using these quantities, MCPN-SPSA updates the parameter $\theta$ along a descent direction as follows: $\forall t \ge 1$, 
\begin{align}
&\theta_{i}(t+1)  =  \Gamma_i \bigg( \theta_{i}(t) -  a(t)\sum\limits_{j = 1}^{N} M_{i, j}(t) \dfrac{\hat J^{\theta(t)+\delta\Delta(t)}(x_0) - \hat J^{\theta(t)-\delta\Delta(t)}(x_0)}{2\delta \Delta_{j}(t)}
 \bigg).\label{eq:hessian-update-rule}
\end{align}

Along similar lines, using \eqref{eq:sf-hessian-estimate}, the SF variant of the above algorithm would update the Hessian estimate as follows:
For all $t \ge 1$, $i,j,k\in \{1,\ldots,N\}$, $j\le k$, update
\begin{align}
H_{i, i}(t + 1) =&  H_{i, i}(t) + a(t) \bigg (\dfrac{\big(\Delta^2_{i}(t)-1\big)}{\delta^2} (\hat J^{\theta(t)+\delta\Delta(t)}(x_0)  + \hat J^{\theta(t)-\delta\Delta(t)}(x_0)) - H_{i, j}(t) \bigg ),\\
H_{j,k}(t + 1) =&  H_{j,k}(t) +a(t) \bigg (\dfrac{\Delta_{i}(t)\Delta_{j}(t)}{\delta^2} (\hat J^{\theta(t)+\delta\Delta(t)}(x_0)   + \hat J^{\theta(t)-\delta\Delta(t)}(x_0)) - H_{j,k}(t) \bigg ),
\end{align}
and for $j > k$, we set $H_{j, k}(t+1) = H_{k, j}(t+1)$. In the above, $\Delta(t)$ is a $N$ vector of independent Gaussian $\N(0,1)$ random variables. As before, form the Hessian estimate matrix $H(t)$ and its inverse $M(t) = \Upsilon(H(t))^{-1}$. Then, the policy parameter $\theta$ is then updated as follows: $\forall t\ge 1$, 
\begin{align}
&\theta_{i}(t+1)  =  \Gamma_i \bigg( \theta_{i}(t) -  a(t)\sum\limits_{j = 1}^{N} M_{i, j}(t)  \Delta_{j}(t)\dfrac{(\hat J^{\theta(t)+\delta\Delta(t)}(x_0) - \hat J^{\theta(t)-\delta\Delta(t)}(x_0))}{2\delta}
 \bigg).\label{eq:sf-hessian-update-rule}
\end{align}

% \begin{remark}
%  A computationally efficient alternative to inverting the Hessian $H$ is to use the Woodbury's identity. See Section \ref{sec:woodbury} of \cite{techreport} for a detailed description.
% \end{remark}
% \todo{I commented out Woodbury variant. Need to put this in techreport}              
\paragraph{Woodbury variant.}
A computationally efficient alternative to inverting the Hessian $H$ is to use the Woodbury's identity.
Woodbury's identity states that \[(A + UCV)^{-1} = A^{-1} - A^{-1} U \left ( C^{-1} + V A^{-1} U \right )^{-1} V A^{-1}\] where $A$ and $C$ are invertible square matrices and $U$ and $V$ are rectangular matrices of appropriate sizes.
Let $U(t) = \dfrac{1}{\delta} \left [ \dfrac{1}{\Delta_1(t)},\dfrac{1}{\Delta_2(t)}, \ldots ,\dfrac{1}{\Delta_{|N|}(t)} \right ]^T$, $V(t) = \dfrac{1}{\delta} \left [ \dfrac{1}{\widehat\Delta_1(t)}, \dfrac{1}{\widehat\Delta_2(t)}, \ldots ,\dfrac{1}{\widehat\Delta_{|N|}(t)} \right ]$ and \\$C(t) = b(t) \left(\hat J^{\theta(t)+\delta\Delta(t) + \delta\widehat\Delta(t)}(x_0) - \hat J^{\theta(t)+\delta\Delta(t)}(x_0)\right)$

Using the Woodbury's identity, MCPN algorithm would update the estimate $M(t)$ of the Hessian inverse as follows: 
  \begin{align}
\label{eq:wudbury-update-rule}
M(t + 1) = &\Upsilon \bigg (\dfrac{M(t)}{1 - a(t)} \bigg [ I -\dfrac{ C(t) U(t) V(t) M(t)}{1 - b(t) + C(t) V(t) M(t) U(t) } \bigg ] \bigg ),
\end{align}
where $M(0)= k I$, with $I$ denoting the identity matrix and $k$ is some positive constant. The update of the policy parameter $\theta(t)$ is the same as before (see \eqref{eq:hessian-update-rule}).

%%%%%%%%%%%%%%%%%%%%%%%%%%%%%%%%%%%%%%%%%%%%
\section{Main Results}
\label{section:results}
\subsection{Analysis of the MFMC estimator.}
\begin{description}
\item[(A1)]We assume that  the dynamics $f$, the cost function  $ c$ and the policies $h^\theta, \forall \theta \in \Theta$ are  Lipschitz continuous, i.e., we assume  that  there exist finite constants $L_f,  L_c$ and, $L^\theta ,  \forall \theta \in \Theta$such that:\\
$\forall~(x, x',u, u', w) \in \mathcal X^2 \times \mathcal U^2 \times \mathcal W,$
\begin{align*}
\hspace{-2em}\|f(x,u,w) - f(x',u',w) \|_{\mathcal X}  &\leq  L_f ( \| x - x' \|_{\mathcal X} + \| u - u' \|_{\mathcal U} ),\\[1ex]
\hspace{-2em}| c(x,u,w) - c(x',u',w) |  &\leq  L_c ( \| x - x' \|_{\mathcal X} + \| u - u' \|_{\mathcal U} ), \\[1ex]
\hspace{-2em} \| h^\theta(x) - h^\theta(x')  \|_{\mathcal U} &\leq  L^\theta \| x - x' \|_{\mathcal X} , \forall \theta \in \Theta,
 \end{align*}
where $\|.\|_{\mathcal X}$ and  $\|.\|_{\mathcal U}$ denote the chosen norms over the spaces $\mathcal X$ and $\mathcal U$, respectively.
\item[(A2)]We suppose  that $\mathcal X \times \mathcal U$ is bounded when  measured using the distance metric  $d$, defined as follows:
\begin{align*}
 d((x,u),(x',u')) =  \| x - x' \|_{\mathcal X} +  \| u - u' \|_{\mathcal U} , \\
\forall (x,x',u,u') \in \mathcal  X^2 \times \mathcal U^2 . 
\end{align*}
\end{description}
%  \todo[inline]{Clarify  sample   $\mathcal   P_n$ is a sample of what?}
% \todo[inline]{make k-dispersion a definition}
  \begin{definition}
  Given $k \in \mathbb N \setminus \left\{ 0 \right\}$ with $k \leq n$, we define the {\it $k-$dispersion},
$\alpha_k(\mathcal   P_n)$:%   of   the   set of state-action pairs   $\mathcal   P_n$:
\begin{align*}
\alpha_{k}(\mathcal  P_n) =  \underset {(x,u)  \in \mathcal  X \times
  \mathcal U} {\sup}  d^{\mathcal P_n}_{k}(x,u)  \ ,
\end{align*}
where $d^{\mathcal P_n}_{k}(x,u)$ denotes the distance of $(x,u)$
to its $k-$th nearest neighbor (using the distance metric $d$) in
the $\mathcal P_n$ sample, where $\mathcal P_n$ denotes the sample of state-action pairs $\mathcal P_n = \{ (x^l, u^l) \}_{l=1}^n$.
\end{definition}
The $k-$dispersion is the smallest radius such that all  $d$-balls  in $\mathcal  X \times  \mathcal U$  of  this radius  contain at least  $k$ elements from ${\mathcal P_{n}}$.
%it can be interpreted as a worst-case measure on how closely ${\cal P}_n$  covers the ${\cal X}\times {\cal U}$ space using the $k$-th nearest neighbors. 
We finally define the expected value of the MFMC estimator:
\begin{definition}[Expected Value  of $\hat J^\theta\left(x_0 \right)$]$\\ $
We denote by $E^\theta_{p,\mathcal P_n}(x_0)$ the  expected value of the MFMC estimator that builds $p$ trajectories: 
\begin{align*}
E^\theta_{p,  {\mathcal P}_n}(x_0)  = \underset  { w^{1},\ldots,w^{n}  \sim    p_{\mathcal    W}(.)}    {\mathbb   E}
\left[    \hat J^\theta\left( x_0 \right)  \right] \ .
\end{align*}
\end{definition}
%\todo[inline]{Define $p$ in the definition}
The following lemma bounds the bias of the MFMC estimator in expectation, while Lemma \ref{lemma_high_proba_MFMC} provides a bound in high-probability.
\begin{lemma} \label{lemma_bias_MFMC}%[Bias Bound for $\hat J^\theta\left( x_0 \right)$]
Under (A1)-(A2), one has:
\begin{align*}
&\left|  J^\theta(x_0) - E^\theta_{p,\mathcal  P_n}(x_0) \right| \leq C^\theta   \alpha_{pT}\left(\mathcal P_n\right) + \frac{\gamma^T}{1- \gamma} \\
&  \mbox{ with  } C^\theta = \frac{L_c}{1 -  \gamma L_f(1+L^\theta) } \sum_{t=0}^{T-1} \gamma^t  \ . \nonumber
\end{align*}
\end{lemma}
\begin{lemma} \label{lemma_high_proba_MFMC}%[High-probability Bound for $\hat J^\theta\left( x_0 \right)$]
Under (A1)-(A2), one has for any $\eta>0$:
\begin{align}
 \left|  J^\theta(x_0) - \hat  J^\theta(x_0)  \right| &\leq \left( C^\theta   \alpha_{pT}\left(\mathcal P_n\right) + \frac{\gamma^T}{1- \gamma} \right) \sqrt{\frac{2 \ln(2 / \eta)}{p}} \nonumber \\
&\leq K_\eta \label{lemma_bias_hp_ineq2}
\end{align}
 with probability at least $1 - \eta$.
In the above, $K_\eta >0 $ is a finite constant  independent of $\theta$.
\end{lemma}
% \todo{I commented out the proof idea, as it not new to this paper. Need to put this in techreport}

\subsection{Analysis of the MCPG algorithm}
\label{sec:conv-mcpg}
In this section, we describe the difficulty in establishing the asymptotic convergence for the MCPG-SPSA algorithm - a difficulty common to all our algorithms.    
An important step in the analysis is to prove that the bias in the MFMC estimator contributes a asymptotically negligible term to the $\theta$-recursion \eqref{eq:spsa-update-rule}. In other words, it is required to show that \eqref{eq:spsa-update-rule} is asymptotically equivalent to the following in the sense that the difference between the two updates is $o(1)$:
\begin{align}
&\theta_i(t+1)  = \Gamma_i \bigg( \theta_i(t) -  a(t) \frac{J^{\theta(t) +\delta\Delta(t)}(x_0) - J^{\theta(t)-\delta\Delta(t)}(x_0)}{2\delta\Delta_i(t)} \bigg).
\label{eq:mcpg-eqv}
\end{align}
As a first step towards establishing this equivalence, we first re-write the $\theta$-update in \eqref{eq:spsa-update-rule} as follows:
\begin{align*}
\theta_i(t+1)  = \Gamma_i \bigg( \theta_i(t) -  a(t) \frac{J^{\theta(t) +\delta\Delta(t)}(x_0) - J^{\theta(t)-\delta\Delta(t)}(x_0)}{2\delta\Delta_i(t)} + a(t)  \xi(t) \bigg),
\end{align*}
where $\xi(t) = \dfrac{\epsilon^{\theta(t) +\delta\Delta(t)} - \epsilon^{\theta(t)-\delta\Delta(t)}}{2\delta\Delta_i(t)}$.
In the above, we have used the fact MFMC returns an estimate $\hat J^{\theta}(x_0) = J^{\theta}(x_0) + \epsilon^\theta$, with $\epsilon^\theta$ denoting the bias. 

Let $\zeta(t) = \sum_{s = 0}^{t} a(s) \xi_{s + 1}$. Then, a critical requirement for establishing the equivalence of \eqref{eq:spsa-update-rule} with \eqref{eq:mcpg-eqv} is the following condition: 
\begin{align}
\sup_{s\ge0} \left (\zeta(t+s) - \zeta(t) \right) \rightarrow 0 \text{ as } t\rightarrow\infty. 
\label{eq:bias-condition}
\end{align}
While the bias $\epsilon^\theta$ of MFMC can be bounded (see Lemmas \ref{lemma_bias_MFMC}--\ref{lemma_high_proba_MFMC}), it is difficult to ensure that the above condition holds. 

Assuming that the bias is indeed asymptotically negligible, the asymptotic convergence of MCPG can be established in a straightforward manner. In particular, using the ordinary differential equation (ODE) approach \cite{borkar2008stochastic}, it can be shown that \eqref{eq:mcpg-eqv} is a discretization (and hence converges to the equilibria) of the following ODE:  
\begin{align}
\label{eq:theta-ode}
\dot{\theta} = \bar{\Gamma}\left ( \nabla_\theta J^\theta(x_0)\right),
\end{align}
where $\bar\Gamma$ is a projection operator that ensures $\theta$ evolving according to \eqref{eq:theta-ode} remains bounded. 

\begin{remark}
The detailed proof of convergence of MCPG as well as other proposed algorithms, under the assumption that \eqref{eq:bias-condition} holds is provided in Appendix \ref{sec:convergence}.
\end{remark}

\section{Numerical Illustration}
\label{section:experiments}
% \subsection{System and setting}
\newcommand{\errorband}[5][]{ % x column, y column, error column, optional argument for setting style of the area plot
\pgfplotstableread[col sep=comma, skip first n=2]{#2}\datatable
    % Lower bound (invisible plot)
    \addplot [draw=none, stack plots=y, forget plot] table [
        x={#3},
        y expr=\thisrow{#4}-2*\thisrow{#5}
    ] {\datatable};

    % Stack twice the error, draw as area plot
    \addplot [draw=none, fill=gray!40, stack plots=y, area legend, #1] table [
        x={#3},
        y expr=4*\thisrow{#5}
    ] {\datatable} \closedcycle;

    % Reset stack using invisible plot
    \addplot [forget plot, stack plots=y,draw=none] table [x={#3}, y expr=-(\thisrow{#4}+2*\thisrow{#5})] {\datatable};
}
% \begin{tikzpicture}
% \begin{axis}[
%     compat=1.5.1,
%     no markers,
%     enlarge x limits=false,
%     ymin=0,
%     xlabel=iterations,
%     ylabel=$\theta$,
%     x post scale=1.2
% ]
% 
% % Northern Hemisphere Average
% \errorband[gray, opacity=0.3]{exp_results/test.csv}{0}{1}{2}
% 
% % Northern Hemisphere 2012
% \addplot [thick, black] table [
%     x index=0,
%     y index=1,
%     col sep=comma,
% ] {exp_results/test.csv};
% 
% \end{axis}
% \end{tikzpicture}

We consider the 1-dimensional system ruled by the following dynamics:
\begin{align*}
f(x, u, w) &= \mbox{sinc}(10*(x + u + w)), \text{ where}\\
% where sinc denotes the sinus cardinal mapping 
\mbox{sinc}(x) &= \sin(\pi x) / (\pi x).
\end{align*}
The cost function is defined as follows:
$$c(x, u, w) = -\frac{1}{2 \pi}\exp\left(- \frac{x^2 + u^2}{2} + w \right).$$
We consider a class of linearly parameterized policies:
$$\mu^\theta(x) = \theta x, \quad \forall \theta \in [0, 1].$$
The disturbances are drawn according to a uniform distribution between $[-\frac{\epsilon}{2}, \frac{\epsilon}{2}]$ with $\epsilon = 0.01$. The initial state of the system is fixed to $x_0 = -1$ and the discount factor is set to $\gamma = 0.95$.  The truncation of artificial trajectories is set to $T= \frac{1}{1 - \gamma} = 20$, and the number of artificial trajectories rebuilt by the MFMC estimator is set to $p = \lceil \ln(n/T) \rceil$.
We give in Figure \ref{fig:expected_returns} a plot of the evolution of the expected return $J^{\theta}(x_0)$ as a function of $\theta$ (obtained through extensive Monte Carlo simulations). We observe that the expected cost-to-go $J^{\theta}(x_0)$ is minimized for values of $\theta$ around 0.06.

 \begin{figure}
    \centering
\hspace{-2em}\tabl{c}{\scalebox{1.0}{\begin{tikzpicture}
      \begin{axis}[xlabel={iterations},ylabel={$J^\theta(x_0)$}]
      \addplot[no markers,blue,thick] table[x index=0,y index=1,col sep=comma,each nth point={1}] {exp_results/plot_return_theta_Jtheta};
      \end{axis}
      \end{tikzpicture}}\\[1ex]}
    \caption{$J^\theta(x_0)$ vs. $\theta$. Note that the global minimum is $\theta_{min} = 0.06$.}
    \label{fig:expected_returns} 
\end{figure}
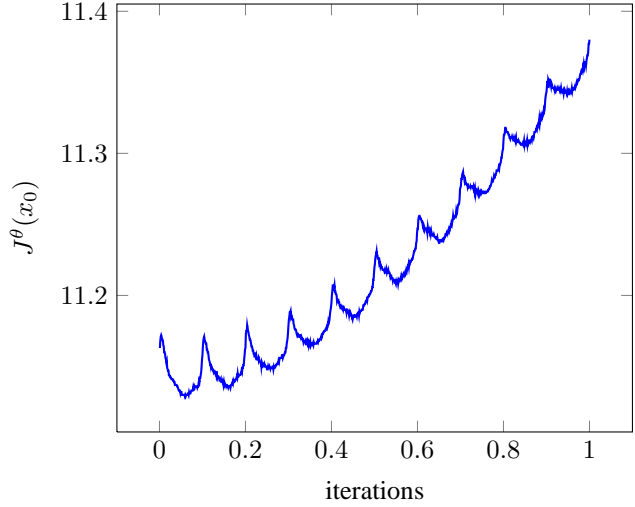

 \begin{figure}
    \centering
    \begin{tabular}{c}
    \subfigure[MCPG-SPSA vs. MCPG-SF]
    {
      \hspace{-2em}\tabl{c}{\scalebox{1.0}{\begin{tikzpicture}
      \begin{axis}[
	xlabel={iterations},
	ylabel={$\theta^{<alg>}(t)$},
       legend entries={
	 ,
        MCPG-SPSA,
        ,
        MCPG-SF
        },
        legend pos=north east,
      ]

      % SPSA
      \errorband[red!50!white, opacity=0.3]{exp_results/smart_csv_mcpg}{0}{1}{2}
      \addplot [thick, red] table [x index=0, y index=1,col sep=comma] {exp_results/smart_csv_mcpg};
      % SPSA
      \errorband[blue!50!white, opacity=0.3]{exp_results/smart_csv_mcpg_smooth}{0}{1}{2}
      \addplot [thick, blue] table [x index=0, y index=1,col sep=comma] {exp_results/smart_csv_mcpg_smooth};
      \end{axis}
      \end{tikzpicture}}\\}
      \label{fig:mcpg} 
    }
     \\
         \subfigure[MCPN-SPSA vs. MCPN-SF]
    {
      \hspace{-2em}\tabl{c}{\scalebox{1.0}{\begin{tikzpicture}
      \begin{axis}[
	xlabel={iterations},
	ylabel={$\theta^{<alg>}(t)$},
       legend entries={
	,
        MCPN-SPSA,
        ,
        MCPN-SF
        },
        legend pos=north east,
      ]

      % SPSA
      \errorband[red!50!white, opacity=0.3]{exp_results/smart_csv_mcpn}{0}{1}{2}
      \addplot [thick, red] table [x index=0, y index=1,col sep=comma] {exp_results/smart_csv_mcpn};
      % SPSA
      \errorband[blue!50!white, opacity=0.3]{exp_results/smart_csv_mcpn_smooth}{0}{1}{2}
      \addplot [thick, blue] table [x index=0, y index=1,col sep=comma] {exp_results/smart_csv_mcpn_smooth};
      \end{axis}
      \end{tikzpicture}}\\}
      \label{fig:mcpn} 
    }
%     \subfigure[50 runs of MCPN]
%     {
%       \hspace{-2em}\tabl{c}{\scalebox{0.75}{\begin{tikzpicture}
%       \begin{axis}[xlabel={iterations},ylabel={$\theta^{i, MCPN}(t), i=1,\ldots,50$}]
%       \foreach \i in {1,...,50}
% 	{
%       \addplot+[no markers] table[x index=0,y index=\i,col sep=comma,each nth point={1}] {exp_results/50_runs_mcpn_n200_500it};
%       }  
%       \end{axis}
%       \end{tikzpicture}}\\}
%       \label{fig:mcpn} 
%     }
%     \\
%     \subfigure[Evolution of the average for MCPG and MCPN]
%     {
%       \hspace{-2em}\tabl{c}{\scalebox{0.75}{\begin{tikzpicture}
%       \begin{axis}[xlabel={iterations},ylabel={$\left(  \frac{1}{50} \sum_{i=1}^{50} \theta^{i, MCPG}(t) \right)_t, t=1,\ldots,500$}]
%       \addplot[no markers,blue,thick] table[x index=0,y index=1,col sep=comma,each nth point={1}] {exp_results/average_mcpg_n200_500it};
%       \addlegendentry{MCPG}
%       \addplot+[no markers,red,thick] table[x index=0,y index=1,col sep=comma,each nth point={1}] {exp_results/average__mcpn_n200_500it};
%       \addlegendentry{MCPN}
%       \end{axis}
%       \end{tikzpicture}}\\}
%       \label{fig:mcpg_mcpn} 
%     }
\end{tabular}
     \caption{Empirical illustration of the MCPG and MCPN algorithms on an academic benchmark.}
    \label{fig:all_results}
\end{figure}
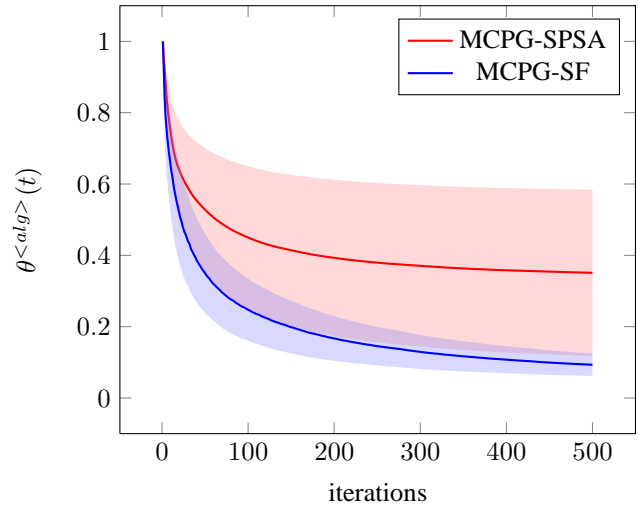
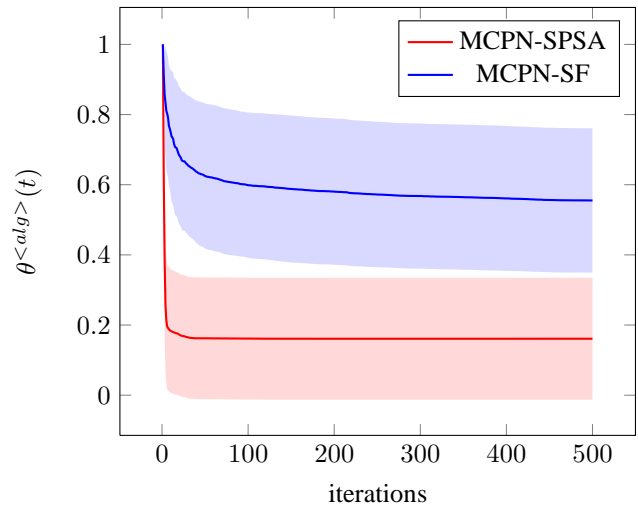

% \begin{figure}
% \begin{center}
% \includegraphics[width=2in]{average_mcpg_mcpn_n_1000.png}
% \caption{\label{fig:mcpg_mcpn} Evolution of the average sequences $\left( \frac{1}{10} \sum_{i=1}^{10} \theta^{i, MCPG}_t \right)_t$ and $\left(  \frac{1}{10} \sum_{i=1}^{10} \theta^{i, MCPN}_t \right)_t$, $t \in \{0, \ldots, N-1 \} $}
% \end{center}
% \end{figure}
% \begin{figure}
% \begin{center}
% \includegraphics[width=2in]{returns_obtained_policies.png}
% \caption{\label{fig:return_of_obtained_policies} Empirical distributions of the expected returns of policies outputted by the MCPG and the MCPN algorithms from  $\mathcal F^{1}_n, \ldots,  \mathcal F^{10}_n$}
% \end{center}
% \end{figure}

% \subsection{Protocole and results}

In order to observe the impact of the randomness of the set of transitions (induced by the disturbances) on the algorithms, we generate 50 samples of transitions $\mathcal F^{1}_n, \ldots,  \mathcal F^{50}_n$, each sample containing $n = 200$ transitions. For each set $\mathcal F^{i}_n, i=1 \ldots 50$, the set of state-action pairs $\mathcal P_n = \{ (x^l, u^l)  \}_{l=1}^{n}$ is the same and generated deterministically from a grid, i.e. $\mathcal P_n =  \{ ( -1 + 2*i/\sigma   ,  -1 + 2*j/\sigma )  \}_{i,j =0}^{\sigma-1}$ with $\sigma = \lfloor  \sqrt{n} \rfloor $. The randomness of each set $\mathcal F^{i}_n$ comes from the disturbances $w^l \quad l=1 \ldots n$ along which transitions are generated.

Then, for each sample $\mathcal F^{i}_n$, we run all the four algorithms - MCPG-SPSA, MCPG-SF, MCPN-SPSA and MCPN-SF -  for $500$ iterations. This generates the sequences $\left( \theta^{i, <alg>}(t) \right)_t$, where $<alg>$ denotes the algorithm. For each algorithm run, we set $\delta = 0.1$ and the step-size $a(t) = \frac{1}{t}$, for all $t$. Further, the operator $\Gamma$ projects $\theta(t)$ into the interval $[0,1]$, while the Hessian operator $\Upsilon$ projects into $[0.1,\infty)$. 

Figure \ref{fig:all_results} presents the average evolution of the parameter sequence in each of the 50 runs for all the algorithms (bands around the average curves represent 95\% confidence intervals). From these plots, we observe that the MCPG-SF approach outperforms the other algorithms on this academic benchmark, with a much lower variance and higher precision.

\section{Extension to Risk-Sensitive Criteria}
\label{section:risk}
The objective here is to minimize the variance of sum of discounted costs in addition to the usual criterion of minimizing the expected cost-to-go $J^{\theta}(x_0)$. Recent work in this direction is by \citep{prashanth2013actor}, where the authors presented actor-critic algorithms. The notable difference here is that, unlike \\\citep{prashanth2013actor}, we use a Monte Carlo like policy evaluator and do not resort to linear function approximation for the value function. Instead, we estimate both the expected and variance of the sum of costs using a MFMC estimator and use it to solve a (constrained) risk sensitive MDP. 
% We make this precise in the following.

Let $R^{\theta}(x_0)$ denote the discounted sum of costs, defined as:
\begin{eqnarray}
R^\theta(x_0) = \sum_{t=0}^{\infty} \gamma^t c(x_t, \mu^\theta(x_t),w_t) \label{eq:risk-random_variable_return}
\end{eqnarray}
with $x_{t+1} =  f(x_t, \mu^\theta(x_t),w_t) $ and $w_t \sim p_{\mathcal W}(\cdot) $. Recall that $J^\theta(x_0)$ is the expectation of this random variable.
Further, let $V^\theta(x_0)$ denote the variance of $R^{\theta}(x_0)$.
The risk-sensitive MDP, which is a constrained optimization problem, is formulated as follows:
\begin{equation}
\label{eq:discounted-risk-measure}
\min_{\theta \in C} J^\theta(x_0) \quad \text{subject to} \quad V^\theta(x_0)\leq\alpha 
\end{equation}
In the above,  $\alpha>0$ is a constant bound on the variance that we would like to achieve. 
Following the technique of \citep{prashanth2013actor}, we relax the above problem as
$\max_\lambda\min_\theta L(\theta,\lambda) \stackrel{\triangle}{=} J^\theta(x_0)+\lambda\big(V^\theta(x_0)-\alpha\big)$, where $\lambda$ denotes the Lagrange multiplier.

\subsection{Risk-sensitive variant of MCPG}
We now describe a variant of MCPG algorithm that solves \eqref{eq:discounted-risk-measure}. 
The MFMC estimator is enhanced to return estimates of both the mean as well variance of the expected cost-to-go. Using these values, the algorithm would update $\theta$ and $\lambda$ using a two timescale procedure as follows - \begin{inparaenum}[\bfseries(i)]                                                                                                      
\item a faster timescale $a(t)$ for gradient descent in the primal for the $\theta$ policy parameter;
\item a slower timescale $b(t)$ for the ascent in the dual for the Lagrange multiplier $\lambda$.
 \end{inparaenum}
 
The variance $V^\theta(x_0)$ can be estimated by combining the costs given by the artificial trajectories with the classical estimator of the variance, as follows:
\begin{eqnarray*}
\hat {V}^{\theta}(x_0) = \frac{1}{p-1} \sum_{i=1}^{p} \left( \sum_{t=0}^{T-1} \gamma^t c^{l^i_t}  - \hat J^\theta(x_0)  \right)^2
\end{eqnarray*}
We now use SPSA estimates of the gradient of the Lagrangian $L(\theta,\lambda)$ to descend in the primal and the sample of the constraint on the variance for the ascent in the Lagrange multipliers (Note: $\nabla_\lambda L(\theta,\lambda) = V^\theta(x_0)-\alpha$).
This results in the following update rule for the risk-sensitive variant of MCPG algorithm:
\begin{align}
\theta_{i}(t+1)  &=  \Gamma_i \bigg( \theta_{i}(t)- \frac{a(t)}{2\delta \Delta_{i}(t)}
\big(\hat J^{\theta(t)+\delta\Delta(t)}(x_0) - \hat J^{\theta(t)-\delta\Delta(t)}(x_0)
 + \lambda(t) (\hat V^{\theta(t)+\delta\Delta(t)}(x_0) - \hat V^{\theta(t)-\delta\Delta(t)}(x_0)\big) \bigg),\nonumber\\
 \lambda(t+1) &= \Gamma_\lambda\bigg[\lambda(t) + b(t)\Big(\hat V^{\theta(t)}(x_0)  - \alpha \Big)\bigg]. \label{eq:spsa-risk-update}
\end{align}
In the above,  $\Gamma_\lambda$ is an operator that projects to $[0,\lambda_{\max}]$, where $0<\lambda_{\max} < \infty$, while $\Gamma(\cdot)$ is the projection operator that was defined in Section \ref{sec:mcpg} for the MCPG algorithm. 
  
%   We establish in an appendix to the paper that the above update converges to a locally risk sensitive policy, i.e., to a local saddle point of the Lagrangian $L(\theta,\lambda)$.

\begin{remark}
As stated by \citep{Fonteneau2013ANOR}, one can also use the MFMC estimator to output a Value-at-Risk (VaR)-like criterion as follows:
Let $b \in \mathbb R$ and $c \in [0, 1[$.
\begin{eqnarray*}
\hat J^{\theta,(b,c)}_{RS}(x_0)  = \left\{  \begin{matrix}  + \infty \qquad &  \mbox{ if } \frac{1} {p} \sum_{i=1}^{p} \mathbb I_{\left\{ \mathbf c^{i} > b \right\}}  > c  \ , \\ \hat J^\theta\left(x_0  \right) \qquad & \mbox{ otherwise} \qquad \qquad  \quad \end{matrix}  \right.
\end{eqnarray*}
where $\mathbf c^{i}$ denotes the cost of the $i-$th artificial trajectory:
\begin{eqnarray*}
\mathbf c^{i} = \sum_{t=0}^{T-1}\gamma^t c^{l^i_t} \ .
\end{eqnarray*}
This VaR-like criterion could also be optimized within the MCPG or MCPN frameworks.
\end{remark}

\subsection{Risk-sensitive variant of MCPN}
We derive a variant of MCPN algorithm that incorporate the risk-related criterion of bounding the variance of the cost\footnote{Recall that MCPN algorithm estimated the gradient/Hessian of $J^\theta(x_0)$ alone, while not considering the variance of the return.}. As before, we use SPSA to estimate the gradient and Hessian of the Lagrangian $L(\theta,\lambda)$. 
The overall update rule of this algorithm that operates on two timescales is as follows:
  \begin{align}
H_{i, j}(t + 1) = & H_{i, j}(t) +\nonumber\\
& \hspace{-5em} a(t) \bigg ( \dfrac{\hat J^{\theta(t)+\delta\Delta(t) + \delta\widehat\Delta(t)}(x_0) - \hat J^{\theta(t)+\delta\Delta(t)}(x_0)
+ \lambda(t) \big(\hat V^{\theta(t)+\delta\Delta(t) + \delta\widehat\Delta(t)}(x_0) - \hat V^{\theta(t)+\delta\Delta(t)}(x_0)\big)}{\delta^2 \Delta_{j}(t) \widehat\Delta_{i}(t)} - H_{i, j}(t) \bigg ),\\
\theta_{i}(t+1)  = & \bar\Gamma_i \bigg( \theta_{i}(t) -  \nonumber \\
&\hspace{-5em}a(t)\sum\limits_{j = 1}^{N} M_{i, j}(t) \dfrac{\hat J^{\theta(t)+\delta\Delta(t)}(x_0) - \hat J^{\theta(t)-\delta\Delta(t)}(x_0) + \lambda(t) \big(\hat V^{\theta(t)+\delta\Delta(t)}(x_0) - \hat V^{\theta(t)-\delta\Delta(t)}(x_0)\big)}{2\delta \Delta_{j}(t)}
 \bigg),\\
& - \lambda(t) \frac{\hat V^{\theta(t)+\delta\Delta(t)}(x_0) - \hat V^{\theta(t)-\delta\Delta(t)}(x_0)}{2\delta \Delta_{i}(t)} \bigg),\nonumber\\
 \lambda(t+1) =& \Gamma_\lambda\bigg[\lambda(t) + b(t)\Big(\hat V^\theta(t)  - \alpha \Big)\bigg]. \label{eq:spsa-n-risk-update}
\end{align}
In the above, $\Gamma_\lambda$ is an operator that projects to $[0,\lambda_{\max}]$, where $0<\lambda_{\max} < \infty$, while $\Gamma(\theta) = (\Gamma_1(\theta_1),\ldots,\Gamma_N(\theta_N))^T$ is a projection operator that ensures $\theta$ is bounded and is the same as that used in the MCPG algorithm.  

\section{Conclusions}

We proposed novel policy search algorithms in a batch, off-policy setting. All these algorithms incorporate simultaneous perturbation estimates for the gradient as well as the Hessian of the cost-to-go vector, since the latter is unknown and only biased estimates are available. We proposed both first order policy gradient as well as second order policy Newton algorithms, using both SPSA as well as SF simultaneous perturbation schemes. We noted certain difficulties in establishing asymptotic convergence of the proposed algorithms, owing to the non-stochastic (and non-zero) bias of the MFMC policy evaluation scheme. As a future direction, we plan to investigate conditions under which the bias of MFMC is asymptotically negligible for the policy search algorithms.

\appendix
\section*{Appendix}
\section{Bias and variance of the MFMC Estimator}
\label{sec:appendix-mfmc}
%Let $h_\theta$ be a given control policy.
%\begin{definition}[Expected Value  of $\hat J^\theta\left(x_0 \right)$]$\\ $
%We denote by $E^\theta_{p,\mathcal P_n}(x_0)$ the  expected value: 
%\begin{eqnarray*}
%E^\theta_{p,  {\mathcal P}_n}(x_0)  = \underset  { w^{1},\ldots,w^{n}  \sim    p_{\mathcal    W}(.)}    {\mathbb   E}
%\left[    \hat J^\theta\left( x_0 \right)  \right] \ .
%\end{eqnarray*}
%\end{definition}
%We have the  following theorem:
%\begin{theorem}[Bias Bound for $\mathfrak M^\theta_p\left(\tilde{\mathcal  F_n},  x_0 \right)$] \label{theorem_bias}
%\begin{eqnarray*}
%&&\left|  J^\theta(x_0) - E^\theta_{p,\mathcal  P_n}(x_0) \right| \leq C^\theta   \alpha_{pT}\left(\mathcal P_n\right) + \frac{\gamma^T}{1- \gamma} \\
%&&  \mbox{ with  } C^\theta = \frac{L_c}{1 +  \gamma L_f(1+L_\theta) } \sum_{t=0}^{T-1} \gamma^t  \ . \nonumber
%\end{eqnarray*}
%\end{theorem}
%

The analysis provided in this section is an extension to the infinite horizon setting of the original analysis of the MFMC estimator \citep{Fonteneau2010AISTATS} which was done for the finite-time horizon setting . The present analysis follows the same structure.

\subsection{Proof of Lemma \ref{lemma_bias_MFMC}}
Let us first introduce the random variable $R^\theta(x_0)$ defined as follows:
\begin{eqnarray}
R^\theta(x_0) = \sum_{t=0}^{\infty} \gamma^t c(x_t, \mu^\theta(x_t),w_t) \label{eq:random_variable_return}
\end{eqnarray}
with $x_{t+1} =  f(x_t, \mu^\theta(x_t),w_t) $ and $w_t \sim p_{\mathcal W}(\cdot) $.
Before giving  the proof of Lemma \ref{lemma_bias_MFMC},  we first give
three  preliminary  lemmas.   Given  a disturbance  sequence  $\Omega  =
\left(\Omega(0), \Omega(1), \ldots  \right)  \in  \mathcal  W^\infty$ and a policy $\mu^\theta$, we  define  the $\Omega$-disturbed         state-action         value         function
$Q^{\theta,\Omega}$  as follows: 
\begin{eqnarray*}
Q^{\theta,\Omega}(x,u)  = c(x,u,\Omega(t))  + \sum_{t=1}^{\infty}
\gamma^t c(x_{t},\mu^\theta(x_{t}),\Omega(t)) 
\end{eqnarray*}
with     $x_{1}    =     f(x,u,\Omega(0))$    and     $x_{t+1}    =
f(x_{t},\mu^{\theta}(x_{t}),\Omega(t)), \forall t  \in \mathbb N$.  Then,  we define the expected return  given $\Omega$ the quantity 
\begin{eqnarray*}
 \mathbb  E  [R^\theta(x_0)  |  \Omega]  =  \mathbb E  [R^\theta(x_0)
|  w_0  = \Omega(0)  ,  w_1  = \Omega(1)  \ldots ]  .
\end{eqnarray*}
  From
there,  we  have  the  following  trivial  result:   $  \forall
(\Omega,x_0)  \in  \mathcal W^\infty \times \mathcal X, $
%+++++++++++++++++++++++++++++++++++++++++++++++++++++++++++++++++++++
\begin{eqnarray}
\mathbb E  [R^\theta(x_0) | \Omega]  =
  Q^{\theta,\Omega}(x_0,\mu^\theta(x_0)) \label{proposition_JfromQ}  \ .
\end{eqnarray}
Then, we have the following lemma.
%+++++++++++++++++++++++++++++++++++++++++++++++++++++++++++++++++++++
\begin{lemma}[Lipschitz Continuity of $Q^{\theta,\Omega}$] \label{Q_lipschitz}
Assume that $L_f (1 + L^\theta) < 1/ \gamma $. Then,  
$ \forall (x,x',u,u') \in   \mathcal  X^2  \times  \mathcal  U^2,$ 
\begin{align*}
 \big|  Q^{\theta,\Omega}(x,u)  -  Q^{\theta,\Omega}(x',u') \big|  \leq
  L^{\theta}_{Q} d((x,u),(x',u')),  \text{ where }L^{\theta}_{Q}  = \frac {L_{c}} {1 - \gamma L_f(1+L^\theta)}. 
\end{align*}
\end{lemma}
\paragraph{Proof of Lemma \ref{Q_lipschitz}}   For the sake  of conciseness, we denote  $ \big|
Q^{\theta,\Omega}(x,u) -  Q^{\theta,\Omega}(x',u') \big|$ by $\Delta^Q.$\\
One has:
\begin{align*}
\Delta^Q      &=      \Big|     Q^{\theta,\Omega}(x,u)      -
Q^{\theta,\Omega}(x',u')  \Big|  \\
 &\leq  \Big|  c(x,u,\Omega(0))  - c(x',u',\Omega(0)) \Big|   \\
  & + 
\gamma \Big|  Q^{\theta,\Omega}(f(x,u,\Omega(0)),\mu^\theta(f(x,u,\Omega(0))))    -
Q^{\mu^\theta,\Omega}(f(x',u',\Omega(0)),\mu^\theta(f(x',u',\Omega(0))))
 \Big|  
\end{align*}
 and  the  Lipschitz continuity  of  $c$  gives
\begin{align*}
&\Delta^Q \leq L_c  d((x,u),(x',u')) + \gamma |  Q^{\theta,\Omega}(f(x,u,\Omega(0)),\mu^\theta(f(x,u,\Omega(0))))  \\
& -
Q^{\mu^\theta,\Omega}(f(x',u',\Omega(0)),\mu^\theta(f(x',u',\Omega(0)))) |  
\end{align*}
Naming $f(x,u,\Omega(0)))$ by $y$ and $f(x',u',\Omega(0)))$ by $y'$, we have:
\begin{align*}
\Delta^Q &\leq L_c  d((x,u),(x',u'))  + \gamma |  c(y,\mu^\theta(y),\Omega(1)) +\gamma Q^{\theta,\Omega}(f(y,\mu^\theta(f(y)), \Omega(1)), \mu^\theta(f(y,\mu^\theta(f(y)), \Omega(1)))   \\
&\quad-  c(y',\mu^\theta(y'),\Omega(1) ) - \gamma
Q^{\theta,\Omega}(f(y',\mu^\theta(y'),\Omega(1)), \mu^\theta(f(y',\mu^\theta(y'),\Omega(1))))  |  
\end{align*}
Using the Lipschitz continuity of $c$, we have
\begin{align}
& \Delta^Q \leq L_c d((x,u),(x',u')) + \gamma L_c \Delta((y,\mu^\theta(y)),(y',\mu^\theta(y'))) \nonumber \\
& +\gamma^2 |  Q^{\theta,\Omega}f(y,\mu^\theta(f(y)), \Omega(1)), \mu^\theta(f(y,\mu^\theta(f(y)), \Omega(1))) \nonumber \\
&-  Q^{\theta,\Omega}f(y',\mu^\theta(f(y')), \Omega(1)), \mu^\theta(f(y',\mu^\theta(f(y')), \Omega(1)))| \label{eq:Lipschitz_rec}
\end{align}
According to the definition of $y$ and $y'$, and using the Lipschitz continuity of $f$ and $\mu^\theta$, we have:
\begin{align*}
d((y,\mu^\theta(y)),(y',\mu^\theta(y'))) &=  \|  y - y' \|_{\mathcal X} + \| \mu^\theta(y) - \mu^\theta(y') \|_{\mathcal U}  \\
& = \|  f(x,u,\Omega(0))) - f(x',u',\Omega(0)))  \|_{\mathcal X}  \\
& +   \|  \mu^\theta(f(x,u,\Omega(0)))) - \mu^\theta(f(x',u',\Omega(0))))  \|_{\mathcal U}\\
&\leq L_f d((x,u),(x',u')) + L^\theta L_f d((x,u),(x',u'))
\end{align*}
Plugging this back in equation \ref{eq:Lipschitz_rec}, we obtain:
\begin{align*}
\Delta^Q &\leq L_c d((x,u),(x',u'))  + \gamma L_c (L_f d((x,u),(x',u')) + L^\theta L_f d((x,u),(x',u'))) \\
 & +\gamma^2 |  Q^{\theta,\Omega}f(y,\mu^\theta(f(y)), \Omega(1)), \mu^\theta(f(y,\mu^\theta(f(y)), \Omega(1))) \\
 &-  Q^{\theta,\Omega}f(y',\mu^\theta(f(y')), \Omega(1)), \mu^\theta(f(y',\mu^\theta(f(y')), \Omega(1)))| \\
& = d((x,u),(x',u')) L_c \left( 1 + \gamma L_f(1+L^\theta)\right)  +\gamma^2 |  Q^{\theta,\Omega}f(y,\mu^\theta(f(y)), \Omega(1)), \mu^\theta(f(y,\mu^\theta(f(y)), \Omega(1))) \nonumber \\
& \quad -  Q^{\theta,\Omega}f(y',\mu^\theta(f(y')), \Omega(1)), \mu^\theta(f(y',\mu^\theta(f(y')), \Omega(1)))|
\end{align*}
By iterating the procedure, and assuming that $L_f (1 + L^\theta) < 1/ \gamma $ we obtain:
\begin{align*}
\Delta^Q &\leq L_c(1 + \gamma L_f(1+ L^\theta) + [\gamma L_f(1+ L^\theta)]^2 + \ldots ) \times d((x,u),(x',u')) \\
& = \frac{L_c}{1 - \gamma L_f (1 + L^\theta)} d((x,u),(x',u'))
\end{align*}
which ends the proof.

Given        a truncated artificial        trajectory        $\tau^i        =
[(x^{l^i_t},u^{l^i_t},c^{l^i_t},y^{l^i_t})]_{t=0}^{T-1} $ we denote by
$\Omega^{i}$ its associated disturbance vector $\Omega^{\tau^i} =
[w^{l^i_0},\ldots,w^{l^i_{T-1}} ]$,  i.e.  the vector made  of the $T$
unknown  disturbances that  affected  the generation  of the  one-step
transitions   $(x^{l^i_t},u^{l^i_t},c^{l^i_t},y^{l^i_t})$.  We give the following
lemma.
%++++++++++++++++++++++++++++++++++++++++++++++++++++++++++++++++++++
\begin{lemma}[Bounds     on      the     expected     return     given
  $\Omega$] \label{lemma_bounds}
  $\forall   i  \in  \{  1 , \ldots,  p \},$  \
\begin{eqnarray*}
  b^{\theta}(\tau^i,x_0)  \leq \mathbb  E  [  R^\theta(x_0)   |  \Omega^{i}
  ]  \leq   a^{\theta}(\tau^i,x_0) \ ,
\end{eqnarray*}
with
\begin{eqnarray}
  && b^{\theta}(\tau^i,x_0) = \sum_{t=0}^{T-1}\gamma^t \big[ c^{l^i_t} - L^\theta_{Q}
  \psi^i_t \big] ) - \frac{\gamma^T}{1 - \gamma} \ , \nonumber \\
  && a^{\theta}(\tau^i,x_0) = \sum_{t=0}^{T-1} \gamma^t \big[ c^{l^i_t}
  +  L^\theta_{Q}  \psi^i_t  \big]  + \frac{\gamma^T}{1 - \gamma} \ , \nonumber \\
  && \psi^i_t = d( (x^{l^i_t},
  u^{l^i_t}),(y^{l^i_{t-1}},\mu^\theta(y^{l^i_{t-1}}))) \ , \forall t \in
  \{ 0 , \ldots, T-1 \} \ , \nonumber \\ 
  &&  y^{l^i_{-1}} = x_0, \forall i \in \{  1 , \ldots, p \} .
  \nonumber 
\end{eqnarray}
\end{lemma}
%+++++++++++++++++++++++++++++++++++++++++++++++++++++++++++++++++++++
\paragraph{Proof of Lemma \ref{lemma_bounds}}  
Let  us first  prove the  lower bound.  With $u_0  =  \mu^\theta(x_0)$, the
  Lipschitz continuity of $Q^{\theta,\Omega^{\tau^i}}$ gives
\begin{eqnarray*}
 | Q^{\theta,\Omega^{i}}(x_0,u_0)  -
  Q^{\theta,\Omega^{i}}(x^{l^i_0},u^{l^i_0})|  \leq  L^\theta_{Q} d((x_0,u_0),(x^{l^i_0}, u^{l^i_0} )) \ .
\end{eqnarray*}
Equation (\ref{proposition_JfromQ}) gives
$$Q^{\theta,\Omega^{i}}(x_0,u_0)  = \mathbb E [ R^\theta(x_0) |  \Omega^{i} ] . $$
Thus,
\begin{eqnarray}
 \big| \mathbb E [ R^\theta(x_0)
  | \Omega^{i} ] - Q^{\theta,\Omega^{i}}(x^{l^i_0},u^{l^i_0})
  \big|
& =&  \big|  Q^{\theta,\Omega^{i}}(x_0,\mu^\theta(x_0))  -
  Q^{\theta,\Omega^{\tau^i}}(x^{l^i_0},u^{l^i_0}) \big| \nonumber  \\
&\leq&   L^\theta_{Q}   d((x_0,\mu^\theta(x_0)),(x^{l^i_0},u^{l^i_0} )) \ .
\end{eqnarray}
It follows that
\begin{eqnarray*}
Q^{\theta,\Omega^{i}}(x^{l^i_0},u^{l^i_0})   -
L^\theta_{Q} \psi^i_0  \leq \mathbb  E [ R^\theta(x_0)  | \Omega^{i} ]  \
. 
\end{eqnarray*}
Then, we know that  
\begin{eqnarray*}
Q^{\theta,\Omega^{i}}(x^{l^i_0},u^{l^i_0})
  =  c(x^{l^i_0},u^{l^i_0},w^{l^i_0})  
 +
\gamma  Q^{\theta,\Omega^{i}}\big(f(x^{l^i_0},u^{l^i_0},w^{l^i_0}),
  \mu^\theta(f(x^{l^i_0},u^{l^i_0},w^{l^i_0}))\big)   
  \    . 
\end{eqnarray*}
By       definition      of      $\Omega^{i}$,       we      have:
$c(x^{l^i_0},u^{l^i_0},w^{l^i_0})       =       c^{l^i_0}$      and
$f(x^{l^i_0},u^{l^i_0},w^{l^i_0}) = y^{l^i_0}$ . From there
\begin{eqnarray*}
Q^{\theta,\Omega^{i}}(x^{l^i_0},u^{l^i_0})     =     c^{l^i_0}    + \gamma
Q^{\theta,\Omega^{i}}(y^{l^i_0},\mu^\theta(y^{l^i_0}))  \ , 
\end{eqnarray*}
and
\begin{eqnarray*}
\gamma Q^{\theta,\Omega^{i}}(y^{l^i_0},\mu^\theta(y^{l^i_0}))  +  c^{l^i_0} -
L^\theta_{Q} \psi^i_0  
\leq  \mathbb E [ R^\theta(x_0) | \Omega^{i}  ] \
.
\end{eqnarray*}
The  Lipschitz continuity of  $Q^{\theta,\Omega^{i}}$ gives\\ 
$ \big|  Q^{\theta,\Omega^{i}}(y^{{l^i_0}},\mu^\theta(y^{l^i_0}))
  -  Q^{\theta,\Omega^{i}}(x^{l^i_1},u^{l^i_1})  \big|  \leq
  L^\theta_{Q}
  d((y^{l^i_0},\mu^\theta(y^{l^i_0})),(x^{l^i_1},u^{l^i_1})) =
  L^\theta_{Q}   \psi^i_1,$ \\
which implies that
$$  L^\theta_{Q}                       \psi^i_1                       \le
  Q^{\theta,\Omega^{i}}(y^{l^i_0},\mu^\theta(y^{l^i_0}))  \ . $$
We therefore have
\begin{eqnarray*}
\gamma Q^{\theta,\Omega^{i}}(x^{l^i_1},u^{l^i_1})     +
c^{l^i_0} - L^\theta_{Q} \psi^i_0 - \gamma L^\theta_{Q} \psi^i_1  \leq \mathbb
E [ R^\theta(x_0)  | \Omega^{i} ] .
\end{eqnarray*}
The proof is completed by  iterating this derivation, and by bounding the uncertainty induced by the truncation, which adds a term $\frac{\gamma^T}{1-\gamma}$ to the bound since the reward function $c$ takes value in $[0,1] $.  The upper bound is proved similarly.
We give a third lemma.
%++++++++++++++++++++++++++++++++++++++++++++++++++++++++++++++++++++++
\begin{lemma} \label{lemma_tightness_bounds} $\forall i \in \{
  1, \ldots, p  \}, $
\begin{eqnarray*}
 a^\theta(\tau^i,x_0)  - b^\theta(\tau^i,x_0)  \leq  2 \left(  C
  \alpha_{pT}(\mathcal P_n)  + \frac{\gamma^T}{1-\gamma} \right)
\end{eqnarray*}
with   $C^\theta   = L^\theta_{Q}  \sum_{t=0}^{T-1} \gamma^t
   \ .$
\end{lemma}
%++++++++++++++++++++++++++++++++++++++++++++++++++++++++++++++++++++++
\paragraph{Proof of Lemma \ref{lemma_tightness_bounds}}
By   construction  of   the   bounds,  one   has  $a^\theta(\tau^i,x_0)   -
b^\theta(\tau^i,x_0)  = \sum_{t=0}^{T-1}  2 \gamma^t L^\theta_{Q}  \psi^i_t  + \frac{2 \gamma^T}{1-\gamma} .$ The
MFMC algorithm chooses $p  \times T$ different one-step transitions to build    the    MFMC   estimator    by    minimizing   the    distance $d((y^{l^i_{t-1}},\mu^\theta(y^{l^i_{t-1}})),(x^{l^i_{t}},u^{l^i_{t}}))$,  so by definition  of the  $k$-sparsity of
the  sample $\mathcal P_n$  with $k  = p  T$, one  has
\begin{eqnarray*}
  \psi^i_t = d((y^{l^i_{t-1}},\mu^\theta(y^{l^i_{t-1}})),(
x^{l^i_{t}},u^{l^i_{t}}))    \leq    d^{\mathcal
  P_n}_{pT}(y^{l^i_{t-1}},\mu^\theta(y^{l^i_{t-1}}))  \leq \alpha_{pT}(\mathcal P_n) \ ,
\end{eqnarray*}
which ends the proof.
%+++++++++++++++++++++++++++++++++++++++++++++++++++++++++++++++++++++

Using those  three lemmas, one can  now compute an upper  bound on the
bias of the MFMC estimator.
\paragraph{Proof of Lemma \ref{lemma_bias_MFMC}}
By definition  of $a^\theta(\tau^i,x_0)$  and $b^\theta(\tau^i,x_0)$, we  have
$$\forall  i \in \{  1, \ldots, p \},  \frac {  b^\theta(\tau^i,x_0) +
  a^\theta(\tau^i,x_0)}  {2}  = \sum_{t=0}^{T-1} \gamma^t c^{l^i_t}  \  .  $$  Then,
according        to        Lemmas        \ref{lemma_bounds}        and
\ref{lemma_tightness_bounds}, we have $\forall i \in \{ 1 , \ldots, p
\} \ ,$\\
\begin{eqnarray*}
 \left| \underset {w^{1} , \ldots , w^{n} \sim  p_{\mathcal W}(.)  }
{\mathbb E} \left[ 
\mathbb E  [R^\theta(x_0) | \Omega^{i} ]  - \sum_{t=0}^{T-1} \gamma^t c^{l^i_t}
\right] \right| &\leq& \underset {w^{1} , \ldots , w^{n} \sim p_{\mathcal
    W}(.) } {\mathbb E} 
\left[ \left| \mathbb E [R^\theta(x_0) | \Omega^{i} ] -
\sum_{t=0}^{T-1} \gamma^t c^{l^i_t} \right| \right] \nonumber \\
&\leq& C^\theta \alpha_{pT}(\mathcal P_n) + \frac{\gamma^T}{1-\gamma} .
\end{eqnarray*}
Thus,
\begin{eqnarray*}
 \left| \frac {1} {p} \sum_{i=1}^{p} \underset { w^{1} , \ldots ,
  w^{n} \sim p_{\mathcal W}(.)}  {\mathbb  E}  \left[  \mathbb  E
[R^\theta(x_0)  |  \Omega^{i}  ]  - \sum_{t=0}^{T-1} \gamma^t  c^{l^i_t}
\right]   \right|  &\leq&  \frac   {1}   {p}\sum_{i=1}^{p}  \left|  \underset  {  w^{1} , \ldots , w^{n} \sim p_{\mathcal  W}(.)  } {\mathbb  E}   \left[  \mathbb  E  [R^\theta(x_0)  |
\Omega^{i}  ]  - \sum_{t=0}^{T-1} \gamma^t c^{l^i_t} \right] \right|  \\
&\leq& C^\theta \alpha_{pT}(\mathcal P_n) + \frac{\gamma^T}{1-\gamma} ,
\end{eqnarray*}
which can be reformulated
\begin{eqnarray*}
 \left| \underset  { w^{1} , \ldots , w^{n} \sim p_{\mathcal W}(.)  } {\mathbb E} \left[
\frac {1}  {p} \sum_{i=1}^{p} \mathbb E [R^\theta(x_0)  | \Omega^{i} ]
\right]  - E_{p,\mathcal P_n}^\theta(x_0)  \right| \leq  C^\theta \alpha_{pT}(\mathcal P_n) + \frac{\gamma^T}{1-\gamma} \  ,
\end{eqnarray*}
since  $  \frac  {1}  {p} \sum_{i=1}^{p}  \sum_{t=0}^{T-1}
\gamma^t  c^{l^i_t} = \hat J^\theta( x_0). $ Since  the MFMC  algorithm  chooses $p  \times  T$ different  one-step transitions,  all  the  $\{  w^{l^i_t}  \}_{i=1,t=0}^{i=p,t=T-1}$  are
i.i.d.  according  to $p_{\mathcal W}(.)$.  For all  $i \in \{1, \ldots, p  \},$  The law  of  total  expectation  gives
\begin{eqnarray*}
\underset {w^{l^i_0},\ldots,w^{l^i_{T-1}}  \sim p_{\mathcal W}(.)}
{\mathbb E} \big[   \underset  {w^{l^i_0},\ldots,w^{l^i_{T-1}}   \sim  p_{\mathcal    W}(.)}  {\mathbb E} [R^\theta(x_0) | \Omega^{i}] \big]  = \underset {w_0,\ldots,w_{T-1} \sim p_{\mathcal  W}(.)}  {\mathbb E} [R^\theta(x_0)] =
J^\theta(x_0) \ .
\end{eqnarray*}
This ends the proof.

\subsection{Proof of Lemma \ref{lemma_high_proba_MFMC}}

One first have the triangle inequality.
\begin{eqnarray*}
 \left|    \hat J^\theta\left( x_0 \right)   - J^\theta(x_0)  \right| \leq   \left|    \hat J^\theta\left(x_0 \right)   -  \frac{1}{p} \sum_{i=1}^{p} \mathbb E \left[  R^{\theta}\left(x_0 \right)  |  \Omega^i  \right]  \right| +   \left|   \frac{1}{p} \sum_{i=1}^{p} \mathbb E \left[  R^{\theta}\left(x_0\right) | \Omega^i \right]    - J^\theta(x_0)  \right|  .
\end{eqnarray*}
From the proof given above, one has the following property: $\forall i \in \{1, \ldots, p\},$
\begin{eqnarray}
 \left|    \sum_{t=0}^{T-1} \gamma^t c^{l^i_t}  -  \mathbb E \left[  R^{\theta}\left(x_0\right) | \Omega^i  \right]  \right|   \leq  C^\theta \alpha_{pT}(\mathcal  P_n) + \frac{\gamma^T}{1 - \gamma}. \label{eqn:interval_R}
\end{eqnarray}
% with $C_\theta = L_c \sum_{t=0}^{T-1} \sum_{i=0}^{T-t-1} \left(L_f(1+L_\theta)\right)^i$.
This immediatly leads to:
\begin{eqnarray*}
 \left|    \hat J^\theta\left( x_0 \right)   -  \frac{1}{p} \sum_{i=1}^{p} \mathbb E \left[ R^{\theta}\left(x_0\right)   |  \Omega^i \right] \right| \leq  C^\theta \alpha_{pT}(\mathcal  P_n) + \frac{\gamma^T}{1 - \gamma} .
\end{eqnarray*}
From Equation \eqref{eqn:interval_R}, we have that each variable $\mathbb E \left[ R^\theta\left(x_0 \right) |  \Omega^{i} \right] $ is contained in the interval 
\begin{eqnarray*}
\left[ \sum_{t=0}^{T-1} \gamma^t c^{l^i_t} - C^\theta \alpha_{pT}(\mathcal  P_n)  - \frac{\gamma^T}{1- \gamma},
 \sum_{t=0}^{T-1} \gamma^t c^{l^i_t} + C^\theta \alpha_{pT}(\mathcal  P_n)  + \frac{\gamma^T}{1 - \gamma}   \right]
\end{eqnarray*}
 of width $2 \left( C^\theta \alpha_{pT}(\mathcal  P_n)  + \frac{\gamma^T}{1 - \gamma} \right) $ with probability one.  Since all $\{ w^{l^{i}_t} \} , i=1 \ldots p, t=0 \ldots T-1$ are i.i.d. from $p_{\mathcal W}(\cdot)$, we can apply the Chernoff-Hoeffding inequality:
\begin{eqnarray*}
 \left|   \frac{1}{p} \sum_{i=1}^{p}\mathbb E\left[  R^{\theta}\left(x_0 \right) | \Omega^{i} \right]    -   J^\theta(x_0)  \right|   =  \left|  \hat J^\theta(x_0)    -   J^\theta(x_0)  \right|   \leq    \left( C^\theta \alpha_{pT}(\mathcal  P_n)  + \frac{\gamma^T}{1- \gamma} \right) \sqrt{\frac{2 \ln(2/\eta)}{p}} 
\end{eqnarray*}
with probability at least $1 - \eta.$. The proof of Equation \ref{lemma_bias_hp_ineq2} is obtained by observing that there exists a constant $C := \sup_\theta C^\theta < \infty$. The existence of $C<\infty$ is ensured by the fact that (i) $\mu^\theta$ is continuously differentiable function of $\theta$ and (ii) $\theta$ evolves within a compact set, so the Lipschitz constant of any policy $\theta$ is finite.

\section{Asymptotic convergence of the policy gradient methods}
\label{sec:convergence}
We make the following assumptions for the analysis:
\begin{description}
 \item[(A3)] The policy $\mu^{\theta}$ is continuously differentiable for any policy parameter $\theta \in \Theta$. 
 \item[(A4)] The underlying Markov chain corresponding to any policy $\theta$ is irreducible and positive recurrent.
 \item[(A5)] The step-size sequence $a(n)$ satisfies \\$\sum \limits_{n = 1}^\infty a(n) = \infty$ and $\sum \limits_{n = 1}^\infty  a(n)^2 < \infty$.
 \item[(A6)] The bias of MFMC satisfies the following condition:
\begin{align*}
\hspace{-16em}   \text{Let }\zeta(t) = \sum_{s = 0}^{t} a(s) \xi_{s + 1},\text{ then }\quad\sup_{s\ge0} \left (\zeta(t+s) - \zeta(t) \right) \rightarrow 0 \text{ as } t\rightarrow\infty. 
\end{align*}
\end{description}
The first assumption is standard in policy gradient RL algorithms, while the second assumption ensures that each state gets visited an infinite number of times over an infinite time horizon. The third assumption above imposes standard stochastic approximation conditions on the step-sizes, while the final assumption ensures that the bias of MFMC is asymptotically negligible. 

\subsection{Analysis of MCPG-SPSA}
Before we proceed with the analysis of MCPG, we re-state the following fact regarding the bias of the estimate returned by MFMC:
Let $\epsilon^\theta$ denote the bias of the MFMC estimate $\hat J^{\theta}(x_0)$, i.e., $ \hat J^{\theta}(x_0) = J^{\theta}(x_0) + \epsilon^\theta$. Then, the bias $\epsilon^\theta$ satisfies the following bound:
     \begin{align}
     \label{eq:mfmc-bias-appendix}
    \forall \theta \in \Theta,   \l \epsilon^\theta \r \le K_\eta \mbox{ with probability at least }1-\eta.
     \end{align}
     for some positive, finite constant $K_\eta$ independent from $\theta$. Fix $\eta>0$ and let $E^\eta$ denote the set of all $\theta$ on which \eqref{eq:mfmc-bias-appendix} holds, i.e., $E^\eta = \{ \theta \in \Theta \mid \l \epsilon^\theta \r \le K_\eta \}$. 
     
We use the ordinary differential equation (ODE) approach (\cite{borkar2008stochastic}) to analyze our algorithms. 
Under (A6), the update rule \eqref{eq:spsa-update-rule} of MCPG can be seen to be asymptotically equivalent to\footnote{the equivalence is in the sense that the difference between the \eqref{eq:spsa-update-rule} and \eqref{eq:arxiv-mcpg-eqv} is $o(1)$.}:
\begin{align}
&\theta_i(t+1)  = \Gamma_i \bigg( \theta_i(t) -  a(t) \frac{J^{\theta(t) +\delta\Delta(t)}(x_0) - J^{\theta(t)-\delta\Delta(t)}(x_0)}{2\delta\Delta_i(t)} \bigg).
\label{eq:arxiv-mcpg-eqv}
\end{align}
The proof of convergence of the first order method MCPG is to a set of asymptotically stable equilibrium points of the following ODE:  
\begin{align}
\label{eq:arxiv-theta-ode}
\dot{\theta} = \bar{\Gamma}\left ( \nabla_\theta J^\theta(x_0)\right).
\end{align}
In the above, $\bar{\Gamma}$ is a projection operator that is defined as follows: For any bounded continuous function $g(\cdot)$,
\begin{align}
\label{eq:proj-ode-operator}
\bar{\Gamma}\big(g(\theta)\big) = \lim\limits_{\tau \rightarrow 0}
\dfrac{\Gamma\big(\theta + \tau g(\theta)\big) - \theta}{\tau}.
\end{align}
The projection operator $\bar{\Gamma}(\cdot)$ is necessary to ensure that $\theta$, while evolving through the ODE (\ref{eq:arxiv-theta-ode}),  stays within the bounded set $\Theta \in \R^{N}$.
Let $\Z=\big\{\theta\in C:\bar\Gamma\big(\nabla J^\theta(x_0)\big)=0\big\}$ denote the set of asymptotically stable equilibria of the ODE~\eqref{eq:arxiv-theta-ode}. The main result regarding the convergence of MCPG is as follows:
\begin{theorem}
\label{thm:arxiv-spsa-theta-convergence}
Under (A1)-(A6), for any $\eta >0$, $\theta(t)$ governed by \eqref{eq:spsa-update-rule} converges to $\Z$ in the limit as $\delta \rightarrow 0$, with probability $1-\eta$.
\end{theorem}

Before proving Theorem \ref{thm:arxiv-spsa-theta-convergence}, we prove that the correctness of the SPSA-based gradient estimate \eqref{eq:gradestimate} in the following lemma\footnote{The proof is given here for the sake of completeness and the reader is referred to Chapter 5 of \cite{Bhatnagar13SR} for an extensive treatment on SPSA based gradient estimation.}:
  \begin{lemma}
   \label{eq:spsa-correct}
Recall that $\Delta = (\Delta_1,\ldots,\Delta_N)^T$ is vector of independent Rademacher random variables. We have 
\begin{align}
\label{eq:arxiv-gradestimate}
\lim_{\delta\rightarrow 0} \dfrac{J^{\theta +\delta\Delta}(x_0) - J^{\theta-\delta\Delta}(x_0)}{2\delta\Delta_i(t)} = \nabla_{i} J(\theta)(x_0).
\end{align}
  \end{lemma}
\begin{proof}
Using a Taylor expansion of $J^{\theta+\delta}(x_0)$ and $J^{\theta-\delta}(x_0)$ around $\theta$, we obtain:
\begin{align}
\label{e1}
J^{\theta(t) + \delta \Delta(t)}(x_0) = J^{\theta(t)}(x_0) + \delta \Delta(t)^T \nabla J^{\theta(t)}(x_0) + O(\delta^2), \\[1ex]
J^{\theta(t) - \delta \Delta(t)}(x_0) = J^{\theta(t)}(x_0) - \delta \Delta(t)^T \nabla J^{\theta(t)}(x_0) + O(\delta^2).\label{e2}
\end{align}
From the above, it is easy to see that 
\begin{align}
&\dfrac{J^{\theta(t) + \delta \Delta(t)}(x_0) - J^{\theta(t) - \delta \Delta(t)}(x_0)}{2 \delta \Delta_i(t)}
- \nabla_i J^{\theta(t)}(x_0)\\
 =&\underbrace{\sum_{j=1,j\not=i}^{N} \frac{\Delta_j(t)}{\Delta_i(t)}\nabla_j J^{\theta(t)}(x_0)}_{(I)} + O(\delta)
\end{align}
Term (I) above is zero since $\Delta$ are Rademacher. So, it is easy to see that the estimate \eqref{eq:arxiv-gradestimate} converges to the true gradient $\nabla J^{\theta(t)}(x_0)$ in the limit as $\delta \rightarrow 0$.
\end{proof}

\begin{proof}({\bf Theorem \ref{thm:arxiv-spsa-theta-convergence}})
In lieu of (A6), it is sufficient to analyse the following equivalent update rule for MCPG on the high-probability set $E^\eta$:
\begin{align*}
&\theta_i(t+1)  = \Gamma_i \bigg( \theta_i(t) -  a(t) \frac{J^{\theta(t) +\delta\Delta(t)}(x_0) - J^{\theta(t)-\delta\Delta(t)}(x_0)}{2\delta\Delta_i(t)}
            \bigg).
\end{align*}
Now, using a standard Taylor series expansion (see Chapter 5 of \citep{Bhatnagar13SR}) it is easy to show that  $\dfrac{J^{\theta +\delta\Delta}(x_0) - J^{\theta-\delta\Delta}(x_0)}{2\delta\Delta_i(t)}$ is a biased estimator of $\nabla_\theta J^\theta(x_0)$, where the bias vanishes asymptotically. In more rigorous terms, we have
\begin{align*}
\frac{J^{\theta +\delta\Delta}(x_0) - J^{\theta-\delta\Delta}(x_0)}{2\delta\Delta_i(t)}
&\longrightarrow_{\beta \rightarrow 0} \nabla_{\theta_i} J^\theta(x_0).
\end{align*}
Thus, Eq.~\ref{eq:spsa-update-rule} can be seen to be a discretization of the ODE~\eqref{eq:arxiv-theta-ode}. Further, $\Z_\lambda$ is an asymptotically stable attractor for the ODE~\eqref{eq:arxiv-theta-ode}, with $J^\theta(x_0)$ itself serving as a strict Lyapunov function. This can be inferred as follows:
\begin{align*}
\dfrac{d J^\theta(x_0)}{d t}  
= \nabla_\theta J^\theta(x_0) \dot \theta
= \nabla_\theta J^\theta(x_0)) \bar\Gamma\big(-\nabla_\theta J^\theta(x_0)\big) < 0.  
\end{align*}
The claim now follows from Theorem 5.3.3, pp. 191-196 of~\citep{kushner-clark}. Note that the final claim holds on $E^\eta$, the set with high probability on which the bias of the MFMC estimator is bounded. 
% rewriting Lemma 1 from \cite{borkar2008stochastic} with the deterministic sequence $\xi(t)$ replacing the martingale difference noise there.
\end{proof}

\subsection{Convergence analysis of MCPN-SPSA}
We establish that policy parameter $\theta$ governed by MCPN algorithm \eqref{eq:hessian-update-rule} converges to the set of asymptotically stable equilibria of the following ODE:  
  
\begin{align}
\label{eq:theta-second-ode}
\dot{\theta} = \bar{\Gamma}\left ( (\nabla_{\theta}^2 J^{\theta}(x_0))^{-1} \nabla_\theta J^\theta(x_0)\right).
\end{align}
In the above, $\bar{\Gamma}$ is as defined in \eqref{eq:proj-ode-operator}. 
Let $\Z=\big\{\theta\in C:\bar\Gamma\big((\nabla_{\theta}^2 J^{\theta}(x_0))^{-1}\big)=0\big\}$ denote the set of asymptotically stable equilibria of the ODE~\eqref{eq:theta-second-ode}.

The main result regarding the convergence of $\theta(t)$ governed by \eqref{eq:hessian-update-rule} is given as follows:

\begin{theorem}
\label{thm:mcpn-theta-convergence}
Under (A1)-(A6), for any $\eta >0$, $\theta(t)$ governed by \eqref{eq:hessian-update-rule} converges to $\Z$ in the limit as $\delta \rightarrow 0$, with probability $1-\eta$.
\end{theorem}

Before we prove Theorem \ref{thm:mcpn-theta-convergence}, we establish that the Hessian estimate $H(t)$ in \eqref{eq:hessian-update-rule} converges almost surely to the true Hessian $\nabla^2_{\theta} J^\theta(x_0)$ in the following lemma.
\begin{lemma}
\label{lemma:spsa-n}
With $\delta \rightarrow 0$ as $t \rightarrow \infty$, for all $i, j \in \{1, \ldots, N \}$, we have the following claims with probability one:
\begin{enumerate}[\bfseries(i)]
\item $\left \| \dfrac{J^{\theta(t) + \delta \Delta(t) + \delta \widehat\Delta(t)}(x_0) - J^{\theta(t)+ \delta \Delta(t)}(x_0)}{\delta^2 \Delta_{i}(t) \widehat\Delta_{j}(t)} - \nabla^2_{i, j} J^{\theta(t)}(x_0) \right \| \rightarrow 0,
$\\[1ex]
%  \item $\left \| \dfrac{L(\theta(t) + \delta \Delta(t) + \delta \widehat\Delta(t), \lambda) - L(\theta(t),\lambda)}{\delta \widehat\Delta^{(i)}(t)} - \nabla_{\theta^{(i)}} L(\theta(t), \lambda) \right \| \rightarrow 0,$\\[1ex]
\item $\left \| H_{i, j}(t) - \nabla^2_{i, j} J^{\theta(t)}(x_0) \right \| \rightarrow 0,
$\\[1ex]
\item $\left \| M(t) - \Upsilon(\nabla^2 J^{\theta(t)}(x_0))^{-1} \right \| \rightarrow 0.
$
\end{enumerate}
\end{lemma}

\begin{proof}
The above claims can be established by employing standard Taylor series expansions. For a detailed derivation, the reader is referred to Propositions 7.12 and Lemmas 7.10 and 7.11 of \citep{Bhatnagar13SR}, respectively.  
\end{proof}

% 
% \begin{theorem}
% \label{thm:spsa-theta-convergence}
% Under (A1)-(A4), for any given $\eta>0$ and $\varepsilon > 0$, there exists $\delta_0 >0$ such that for all $\delta \in (0, \delta_0)$, $\theta(t) \rightarrow \Z^\varepsilon$ with probability $1-\eta$, where $\Z=\big\{\theta\in C:\bar\Gamma\big(\nabla J^\theta(x_0)\big)=0\big\}$ denotes the set of asymptotically stable equilibria of the ODE~\eqref{eq:theta-second-ode} and $\Z^\varepsilon=\big\{\theta\in C:||\theta-\theta_0||<\varepsilon,\theta_0\in \Z\big\}$ is its $\varepsilon$-neighborhood.  
% \end{theorem}

\begin{proof}({\bf Theorem \ref{thm:mcpn-theta-convergence}})
 As in the case of the first order method, we can use (A6) to arrive at the following update rule equivalent of the policy parameter $\theta$ on the high-probability set $E^\eta$ :
  \begin{align}
\label{eq:hessian-update-rule1}
H_{i, j}(t + 1) &=  H_{i, j}(t) +  a(t) \bigg ( \dfrac{J^{\theta+\delta\Delta + \delta\widehat\Delta}(x_0) - J^{\theta+\delta\Delta}(x_0)}{\delta^2 \Delta_{j}(t) \widehat\Delta_{i}(t)} - H_{i, j}(t) \bigg ),\\
\theta_{i}(t+1)  &=  \bar\Gamma_i \bigg( \theta_{i}(t) + a(t)\sum\limits_{j = 1}^{N} M_{i, j}(t) \dfrac{J^{\theta-\delta\Delta}(x_0) - J^{\theta+\delta\Delta}(x_0)}{\delta \widehat\Delta_{j}(t)}
 \bigg),
\end{align}

In lieu of Lemma \ref{lemma:spsa-n}, it can be seen that $H_{i,j}(t)$ converges to the true Hessian $\nabla_{\theta_i}^2 J^{\theta}(x_0)$ as $\delta \rightarrow 0$. Thus, the $\theta$-recursion above is equivalent to the following on $E^\eta$:
\begin{align}
 \theta_{i}(t+1)  &=  \bar\Gamma_i \bigg( \theta_{i}(t) + a(t) (\nabla_{\theta_i}^2 J^{\theta}(x_0))^{-1} \nabla_{\theta_i} J^{\theta}(x_0)\bigg).
\end{align}
The above can be seen as a discretization of the ODE \eqref{eq:theta-second-ode}.
Thus, the $\theta(t)$ governed by \eqref{eq:hessian-update-rule} can be seen to converge to a set containing the asymptotically stable equilibria of the above ODE, albeit with probability $1-\eta$ for any $\eta>0$. 
\end{proof}

\subsection{Analysis of SF-based algorithms - MCPG-SF and MCPN-SF}
One can prove SF variants of Theorems \ref{thm:arxiv-spsa-theta-convergence} and \ref{thm:mcpn-theta-convergence} along similar lines, using the following lemma: Recall that $\Delta$ is a $N$-vector of independent Gaussian $\N(0,1)$ random variables for SF-based algorithms. 
\begin{lemma}
\label{lemma:sf-n}
With $\delta \rightarrow 0$ as $t \rightarrow \infty$, for all $i, j \in \{1, \ldots, N \}$, we have the following claims with probability one: (The expectations in the following are w.r.t. the distribution of perturbation random variables $\Delta$)
\begin{enumerate}[\bfseries(i)]
\item $\left \| \E\left[\dfrac{\Delta_i}{\delta} 
\left( J^{\theta+\delta\Delta}(x_0) -  J^{\theta-\delta\Delta}(x_0) \right )\right] - \nabla_{i} J^{\theta}(x_0) \right \| \rightarrow 0$,
\item $\left \| E \left[\dfrac{1}{\delta^2}
\bar{H}(\Delta)(J^{\theta + \delta \Delta}(x_0) + J^{\theta - \delta \Delta}(x_0))\right] - \nabla^2_{i, j} J^{\theta}(x_0) \right \| \rightarrow 0,
$\\[1ex]
%  \item $\left \| \dfrac{L(\theta(t) + \delta \Delta(t) + \delta \widehat\Delta(t), \lambda) - L(\theta(t),\lambda)}{\delta \widehat\Delta^{(i)}(t)} - \nabla_{\theta^{(i)}} L(\theta(t), \lambda) \right \| \rightarrow 0,$\\[1ex]
% \item $\left \| H_{i, j}(t) - \nabla^2_{i, j} J^{\theta(t)}(x_0) \right \| \rightarrow 0,
% $\\[1ex]
% \item $\left \| M(t) - \Upsilon(\nabla^2 J^{\theta(t)}(x_0))^{-1} \right \| \rightarrow 0.
% $
\end{enumerate}
\end{lemma}

\begin{proof}
The proofs of the above claims follow from Propositions 6.5 and 8.10 of \cite{Bhatnagar13SR}, respectively.
\end{proof}  
%%%%%%%%%%%%%%%%%%%%%%%%%%%%%%%%%%%%%%%%%%%%%%%%%%%%%%%%%%%%%%
%%%%%%%%%%%%%%%%%%%%%%%%%%%%%%%%%%%%%%%%%%%%%%%%%%%%%%%%%%%%%%
%%%%%%%%%%%%%%%%%%%%%%%%%%%%%%%%%%%%%%%%%%%%%%%%%%%%%%%%%%%%%%

%%%%%%%%%%%%%%%%%%%%%%%%%%%%%%%%%%%%%%%%%%%%%%%%%%%%%%%%%%%%%%
%%%%%%%%%%%%%%%%%%%%%%%%%%%%%%%%%%%%%%%%%%%%%%%%%%%%%%%%%%%%%%
% \section{Extension to Risk-Sensitive Criteria}
% \label{sec:appendix-risk}
%%%%%%%%%%%%%%%%%%%%%%%%%%%%%%%%%%%%%%%%%%%%%%%%%%%%%%%%%%%%%%
%%%%%%%%%%%%%%%%%%%%%%%%%%%%%%%%%%%%%%%%%%%%%%%%%%%%%%%%%%%%%%

\bibliography{all}
\bibliographystyle{plainnat}

\end{document}